\documentclass[a4paper,reqno]{amsart}

\pdfoutput=1

\usepackage{changes}
\usepackage{amscd,amsthm,amsmath,amssymb,mathtools}
\usepackage{upgreek,bm}
\usepackage{verbatim}
\usepackage{color}
\usepackage{hyperref}
\usepackage{url}
\usepackage{enumerate,enumitem}

\usepackage{graphicx}
\usepackage{epstopdf,epsfig,subfigure}
\usepackage{curve2e}

\usepackage{array}

\usepackage{algorithm}
\usepackage{algorithmic}

\usepackage[numbers,sort&compress]{natbib}

\usepackage{setspace}
\theoremstyle{plain}

\newtheorem{theorem}{Theorem}[section]

\newtheorem{lemma}[theorem]{Lemma}

\newtheorem{assumption}[theorem]{Assumption}

\theoremstyle{definition}
\newtheorem{definition}[theorem]{Definition}
\newtheorem{remark}[theorem]{Remark}
\newtheorem{example}[theorem]{Example}

\numberwithin{equation}{section}

\usepackage{geometry}

\newcommand{\linspan}{\mathop{\rm span}\nolimits}

\newcommand{\rest}{\left.\kern-2\nulldelimiterspace\right|_}
\newcommand{\norm}[2]{\left|#1\right|_{#2}}
\newcommand{\dnorm}[2]{\left\|#1\right\|_{#2}}

\newcommand{\Id}{{\mathbf1}}
\newcommand{\indf}{1}

\newcommand{\p}{\partial}

\newcommand{\clA}{{\mathcal A}}

\newcommand{\clE}{{\mathcal E}}

\newcommand{\clI}{{\mathcal I}}

\newcommand{\clL}{{\mathcal L}}
\newcommand{\clM}{{\mathcal M}}

\newcommand{\clS}{{\mathcal S}}
\newcommand{\clT}{{\mathcal T}}

\newcommand{\bbN}{{\mathbb N}}

\newcommand{\bbR}{{\mathbb R}}

\newcommand{\bfR}{{\mathbf R}}

\newcommand{\bfT}{{\mathbf T}}

\newcommand{\fkC}{{\mathfrak C}}

\newcommand{\fkL}{{\mathfrak L}}

\newcommand{\fkS}{{\mathfrak S}}

\newcommand{\rmD}{{\mathrm D}}

\newcommand{\bfn}{{\mathbf n}}

\newcommand{\bft}{{\mathbf t}}

\newcommand{\rmd}{{\mathrm d}}
\newcommand{\rme}{{\mathrm e}}
\newcommand{\rmf}{{\mathrm f}}
\newcommand{\rmg}{{\mathrm g}}

\newcommand{\rml}{{\mathrm l}}

\newcommand{\rmo}{{\mathrm o}}

\newcommand{\tte}{{\mathtt e}}

\definecolor{DarkBlue}{rgb}{0,0.08,0.45}
\definecolor{DarkRed}{rgb}{.65,0,0}
\definecolor{applegreen}{rgb}{0.55, 0.71, 0.0}

\newcounter{mymac@matlab}
\newcommand{\matlab}{MATLAB%
   \ifnum\value{mymac@matlab}<1%
   \textregistered%
   \setcounter{mymac@matlab}{1}%
   \fi%
  }

\providecommand{\argmin}{\operatorname*{argmin}}

\newcommand{\bfPi}{{\mathbf \Pi}}

\begin{document}
\title{Stabilization of uncertain linear dynamics: an offline-online strategy}
\author{Philipp A.~Guth$^{\tt1}$,  Karl Kunisch$^{\tt 2}$, and S\'ergio S.~Rodrigues$^{\tt1}$}
\thanks{
\vspace{-1em}\newline\noindent
{\sc MSC2020}: 93C40, 93B52, 49N10, 93B51.
\newline\noindent
{\sc Keywords}:  model parameter uncertainty, feedback adaptive control, stabilization.
\newline\noindent
$^{\tt1}$ Johann Radon Institute for Computational and Applied Mathematics,
  \"OAW, %
  Altenbergerstrasse~69, 4040~Linz, Austria.\newline\noindent
$^{\tt2}$   Institute of Mathematics and Scientific Computing,    Karl-Franzens University of Graz,	     	
Heinrichstrasse~36, 8010 Graz, Austria, and Johann Radon Institute for Computational and Applied Mathematics,
  \"OAW, %
  Altenbergerstrasse~69, 4040~Linz, Austria.\newline\noindent
{\sc Emails}:
{\small\tt   philipp.guth@ricam.oeaw.ac.at,\quad karl.kunisch@uni-graz.at,\quad\\ \hspace*{3.4em}sergio.rodrigues@ricam.oeaw.ac.at}
 }

\begin{abstract}
A strategy is proposed for adaptive stabilization of  linear systems, depending on an uncertain parameter. Offline, the Riccati stabilizing feedback input control operators, corresponding to parameters in a finite training set of chosen candidates for the uncertain parameter, are solved and stored in a library. A uniform partition of the infinite time interval is chosen. In each of these subintervals, the input is given by one of the stored parameter dependent Riccati feedback operators.  This  parameter is updated online, at the end of each  subinterval, based on input and output data, where the true data, corresponding to the true parameter, is compared to fictitious data that one would obtain in case the parameter was in a selected subset of the training set. The auxiliary data can be computed in parallel, so that the parameter update can be performed in real time.
The focus is put on the case that the unknown parameter is constant and that
the free dynamics is time-periodic.
The stabilizing performance of the input obtained by the proposed strategy is illustrated by  numerical simulations, for both constant and switching parameters.

\end{abstract}

\maketitle


\pagestyle{myheadings} \thispagestyle{plain} \markboth{\sc P.A. Guth,  K. Kunisch, and S.S. Rodrigues}
{\sc Stabilization of uncertain linear dynamics: an offline-online strategy}



\section{Introduction}
The evolution of a given real-world phenomenon is often described by a mathematical model. Uncertainty is ubiquitous in such modeling tasks.  Motivated by this situation, we consider a mathematical model given by a control system as
\begin{equation}\label{sys-intro}
\dot y=\clA_\sigma y +Bu,\quad y(0)=y_0,\qquad z=Cy,
\end{equation}
for time~$t>0$,  with state~$y(t)\in H$ evolving in a pivot Hilbert space~$H$, and where the dynamics is depending on an uncertain parameter~$\sigma\in\bbR^s$. The linear operator~$\clA_\sigma\colon V\to V'$ defining the free dynamics is time-periodic,  $\clA_\sigma(t)=\clA_\sigma(t+\rho)$, with period~$\rho=\rho(\sigma)>0$, and where~$V\subseteq H$ is another Hilbert space with continuous dual~$V'$. The control input is~$u(t)\in\bbR^{m}$, and the linear control operator~$B\colon\bbR^m\to H$ is independent of time. The initial state~$y_0\in H$ is given. The vector~$z$ represents the output of $p$~sensor measurements, with linear observation operator~$C\colon H\to\bbR^{p}$.
Above, $\dot y\coloneqq\tfrac{\rmd}{\rmd t}$ stands for the time derivative and~$m$, $p$, and~$s$ are fixed positive integers. We focus the exposition having in mind the case of infinite-dimensional systems of parabolic type, where the inclusion~$V\subset H$ is dense continuous and compact. The arguments are valid for the case of finite-dimensional systems as well, with~$V=H=\bbR^n$ for some positive integer~$n$.

In general, it is not easy to find a stabilizing feedback control operator which is able to stabilize the system for all possible values of the parameter~$\sigma\in\bbR^s$, even if we know how to compute a stabilizing feedback for every given~$\sigma$. For example, it may be not  possible/easy to identify a worst-case scenario.
Recently, in~\cite{GuthKunRod23-arx} a feedback stabilizing feedback has been proposed, based on a finite ensemble/sequence~$\varSigma=(\sigma_i)_{i=1}^N$ of training/selected parameters, which shows rather interesting robustness properties, being able to stabilize the system for parameters in a neighborhood of~$\varSigma$. The extension of the strategy in~\cite{GuthKunRod23-arx} to nonautonomous dynamics is work in progress. A direct application of the approach in~\cite{GuthKunRod23-arx} to the case of a large number~$N$ of parameters can be a nontrivial task, even for autonomous finite-dimensional systems where it requires solving for the solution~$\bfPi_\varSigma\in\bbR^{nN\times nN}$ of an algebraic  Riccati equation.
 In this manuscript, to circumvent this issue, we propose an adaptive alternative approach as follows.
We consider another operator~$Q\colon H\to H_Q$, where~$H_Q$ is another Hilbert space.
For the operators~$B$ and~$Q$, we shall assume that,
\begin{equation}\label{assum-StabDete}
(\clA_\sigma,B)\mbox{ is stabilizable\quad and\quad } (\clA_\sigma,Q)\mbox{ is detectable},\quad\mbox{for every~$\sigma\in\fkS$}.
 \end{equation}

Here the matrices~$\clA_\sigma$ are time-$\rho$ periodic, with period~$\rho=\rho(\sigma)$. That is, the time period itself can be subject to uncertainty.

 Now, denoting by~$D^*$ the adjoint of an operator~$D$, we propose to solve for the solution~$\Pi_\sigma(t)\succeq0$ (i.e., positive semidefinite) for each one of the~$N$ time-$\rho$ periodic Riccati  equations (cf.~\cite{DaPratoIchikawa90})
 \begin{equation}\label{Ricc-sig-feed}
 \dot\Pi_{\sigma}+\clA_{\sigma}^*\Pi_{\sigma}+\Pi_{\sigma} \clA_{\sigma}-\Pi_{\sigma}BB^*\Pi_{\sigma}+Q^* Q=0,\qquad\Pi_{\sigma}(t+\rho(\sigma))=\Pi_{\sigma}(t),
 \end{equation}
 for each parameter~$\sigma=\sigma_i\in\varSigma$ and store the corresponding  feedback control input operators~$K_{\sigma_i}=-B^*\Pi_{\sigma_i}\in\bbR^{m\times n}$. Here,~$n$ is the dimension of a discretization of the  possibly infinite-dimensional system~\eqref{sys-intro}.
  These tasks are done offline. Then, online, we apply the feedback input control~$u(t)=K_{\sigma_{i(j)}}y(t)$ corresponding to one of the selected parameters~$\sigma_{i(j)}\in\varSigma$, for time~$t\in I_j^\tau\coloneqq((j-1)\tau,j\tau)$, for a fixed finite time horizon~$\tau>0$, and for positive integers~$j$. This parameter is updated online, by a data-fitting strategy,  based on a comparison involving the input and output data (hereafter, IO data) to the analogue data that we would have if the true parameter were in a subset~$\varSigma^*=\varSigma^*_j\subseteq\varSigma$ of selected parameters. Here,~$\varSigma^*$ may be chosen differently in each time interval~$I_j^\tau$.
 We assume that, at initial time, we have no information on the true parameter~$\sigma$. The online update is necessary, because it is unlikely that~$\clA_\sigma+BK_{\sigma_i}$ is stable for all of the stored~$K_{\sigma_i}$.

 For the comparison, we need to solve the dynamical systems for each parameter in~$\varSigma^*$. This can be done in parallel, so that the parameter update and control input can be computed in real time. The reasons we consider taking subsets~$\varSigma^*_j\subseteq\varSigma$ is that either because the set~$\varSigma$ may be very large or because we may know (say, after some time) that some parameters in~$\varSigma$ are more meaningful/suitable than others.

 The strategy in~\cite{GuthKunRod23-arx}  provides us with an a priori feedback working for all parameters in~$\fkS$ (under some conditions on~$\fkS$), thus it has the advantage that no auxiliary systems have to be solved online. On the other hand, as we mentioned above, it requires solving, offline, a high-dimensional Riccati equation, which can be a nontrivial task.

An offline-online strategy is also proposed in~\cite{KramerPeherWillc17}. Here, the strategy is different both in the offline and online parts. Actually, the strategy in~\cite{KramerPeherWillc17} is proposed for a more general case of switching (piecewise constant) parameters~$\sigma(t)\in\bbR^s$. Our strategy is also able to respond to changes in the true parameter, as we shall see later on (see section~\ref{sec:switching}). In any case, in general, the obtained input is guaranteed to be stabilizing only if the length of the intervals of constancy (also referred to as {\em dwell time}) is large enough. This is due to possible destabilizing effects of switching signals/parameters (cf.~\cite{AkarPaulSafoMitra06,LiberzonMorse99}) rather than to inaccurate parameter estimates regardless of the utilized strategy.

In~\cite{KramerPeherWillc17}  an extra online step is performed to learn a reduced order model representative of the true dynamics through a minimization problem, and to compute the corresponding Riccati feedback. This demands a fast learning algorithm and a reduced order model in order to be able to perform fast computations (ideally, in real time).
Here, we take control inputs given by the feedback operators stored in a library and do not address such a model reduction task, at least, not beyond proposing to solve the auxiliary systems on a coarser discretization.

Another difference is that in~\cite{KramerPeherWillc17} the uncertain parameter is updated by
comparing some components of the solution selected through a  ``random'' selection operator~$\clS$; see~\cite[sect.~2.3.3]{KramerPeherWillc17}. This comparison is performed at each discrete time instant. The choice of the operator~$\clS$ is mentioned as an important and challenging task. In our approach, we propose to update the parameter, after a predefined time interval rather than at each time-instance, by comparing the available IO data in that time interval. Here it is important to know if such a comparison is sufficient to obtain a parameter update good enough to  guarantee the construction of a  stabilizing control input. We discuss this sufficiency in section~\ref{sS:rmks-Conj} and present supporting simulations in section~\ref{S:numerics}. 

Furthermore,  for a given constant parameter~$\sigma$, in~\cite[sect.~2.1]{KramerPeherWillc17} the dynamics is autonomous (time-invariant), while  here the dynamics can be nonautonomous (time-periodic). A comparison based on nonempty time intervals is likely necessary in the case of nonautonomous systems; indeed,  in this case, the examples in~\cite{Wu74} show that ``stability properties/arguments'' at/from each frozen time instant are meaningless concerning the asymptotic behavior of the solutions..

The task of looking for an update for the uncertain parameter has relations with the inverse problem of parameter identification.  Though our strategy may indeed lead to updates close to the uncertain parameter~$\sigma$, this cannot be guaranteed. But, it will provide a good enough ``identification'' of the dynamics, in order to obtain a stabilizing input.

 The  identification of a parameter~$\sigma$ is addressed, for example, in~\cite{VillaEvanChapBang18}, from the output alone, of the true dynamics. Such  identification may require the use of particular ``sufficiently exciting'' input controls. Our approach is different: instead of looking at the response of the output to appropriately constructed inputs, we use only one input, namely, the input given by one of the feedbacks stored in the library. Additionally, we  construct auxiliary IO data for fictitious dynamics, corresponding to a selected subset of training parameters. Roughly speaking, this auxiliary data will give us additional information to compensate for a possibly  not ``sufficiently exciting'' input.

Adaptive control is sometimes concerned with tracking the output of a reference stable system~\cite{NarendraHan11}, say of a targeted system corresponding to one realization of the uncertain parameter. Here, we do not use such target system, and our goal concerns stabilization (to zero).  Still, in~\cite{NarendraHan11} the strategy also involves the estimation of the uncertain parameter in order to compute the appropriate input.  We do not know whether a strategy as in~\cite{NarendraHan11} can be applied to the entire class of systems we address here, namely, the systems in~\cite{NarendraHan11} are assumed in ``companion form'' and it is not clear for us what an analogue of this form would be either in the setting of infinite-dimensional systems, as parabolic equations, or in a general time-periodic setting. The technique in~\cite{NarendraHan11}  aims at an estimate of the uncertain parameter  through a dynamical equation (see also~\cite{KrsticSmyshlyaev08} for a class of autonomous parabolic equations, with a particular class of controllers). Here, we aim at updating the parameter using a finite number of training parameters only, and use corresponding feedback input control operators stored in the library.

\subsection{Additional related literature}
Real-world phenomena are complex and mathematical models often contain simplifications, either because some class of dynamics (e.g., linear) are easier to study or because some features (e.g., nonlinear dynamics) are difficult to be fully captured, or simply due to the fact that measurements made to collect data for modeling are inherently subject to small errors. Hence, it is of paramount importance to design control strategies which are robust against  model uncertainties.

 In the context of open-loop control, this problem has been studied recently, for (spatial discretizations of) parabolic equations, for finite time-horizon, in~\cite{GuthKaarnSchilSlo22-arx,MartinezFKessMunPer16,KunothSchwab13}, and for infinite time-horizon receding horizon control in~\cite{AzmiHerrKun23-arx}.

 In the context of closed-loop (feedback) control, besides the above mentioned~\cite{GuthKunRod23-arx,KramerPeherWillc17}, we find also~\cite{Yedavalli14} where the stabilizing property of a nominal Riccati  feedback is investigated under system perturbations; \cite{ChittaroGauth18}, where stabilizability is investigated for an ensemble of Bloch equations; and~\cite{Ryan14}, where a bilinear stabilizing feedback is constructed for an ensemble of oscillators.
Finally, though we focus on stabilizability, we would like to mention selected works addressing the closely related problem of controllability~\cite{LazarLoheac22,HelmkeScho14,DanhaneLoheJung22-hal,CoulsonGhaMans19,Zuazua14}, and the chapters~\cite[Ch.~5]{Lions88}, \cite[Ch.~11.3]{TucsnakWeiss09}.

 \subsection{Strategy and main result}
We denote by~$\bbR$ the set of real numbers and by~$\bbN$ the set of nonnegative integers. Then, by~$\bbR_+=(0,+\infty)$ and~$\bbN_+\coloneqq\bbN\setminus\{0\}$ their subsets of positive numbers.
 We assume that our uncertain parameter vector~$\sigma$ is an element of a compact subset~$\fkS\subset\bbR^s$
 and that~\eqref{assum-StabDete} holds true. In a first step we choose a finite subset~$\varSigma$ of training parameters,
\begin{equation}\label{varSigma}
\varSigma=\{\sigma_i\mid 1\le i\le N\}\subset \fkS,\qquad\sigma_i\ne\sigma_j\quad\mbox{for}\quad i\ne j,
\end{equation}
Then, we compute the time-$\rho$ periodic solution~$\Pi_{\sigma_i}\succeq 0$ of the  differential Riccati equation~\eqref{Ricc-sig-feed}, for each~$\sigma_i\in\varSigma$, and store the corresponding input feedback operators~$K_{\sigma_i}(t)=-B^\top\Pi_{\sigma_i}(t)\in\bbR^{m\times n}$, for~$t\in[0,\rho)$, in a library
 \begin{equation}
K_\varSigma\coloneqq\{K_{\sigma_i}\mid 1\le i\le N\}.
\end{equation}
The idea is to use the feedback~$\Pi_{\sigma_i}$,  by taking  the~${\sigma_i\in\varSigma}$ closest to~${\sigma}$. Of course, we cannot do this directly because we do not know~$\sigma$. Looking at such~${\sigma_i\in\varSigma}$ as a guess~$\sigma^\tte_{0}=\sigma_i$ for the true~$\sigma$, we shall update~$\sigma^\tte_{0}$ as follows. At initial time~$t=0$, we take an arbitrary guess~$\sigma^\tte_{0}\in\varSigma$. Since this choice may be not necessarily associated with a stabilizing feedback, we shall update it based on IO data comparison. We fix a finite time-horizon~$\tau>0$
and denote the time interval
\begin{equation}\label{timeint-Iktau}
I_j^\tau\coloneqq((j-1)\tau,j\tau),\qquad j\in\bbN_+,
\end{equation}
where we consider the true system with a stored feedback input operator~$K_{\varsigma}$, $\varsigma\in\varSigma$,
\begin{align}
\label{sys-Feed-intro}
&&\dot y&=\clA_\sigma  y +B K_{\varsigma}y, &y(t_0)&=y_{t_0},&& t\in I_j^\tau,
\intertext{and also the auxiliary systems, under the same feedback, for parameters~$\sigma_i\in\varSigma$,}
\label{sys-Feed-intro-aux}
&&\dot y^i&=\clA_{\sigma_i}  y^i +B K_{\varsigma}y^i, &y(t_0)&=y_{t_0},&& t\in I_j^\tau.
\end{align}
To initialize we take~$(j,t_0,y_{t_0},\varsigma)=(1,0,y_0,\sigma_0^\tte)$. We let the true system~\eqref{sys-Feed-intro} operate and store its input~$u_{\sigma}\coloneqq K_{\sigma^\tte_{ {j-1}}}y$ and output~$z_{\sigma}\coloneqq Cy$.
 In parallel, we solve the auxiliary systems~\eqref{sys-Feed-intro-aux} and store the corresponding inputs~$u_{\sigma_i}\coloneqq K_{\sigma^\tte_{ {j-1}}}y^i$ and outputs~$z^i\coloneqq Cy^i$ for all $1\le i \le N$.

Next, we compare the stored IO data in order to construct an update~$\sigma^\tte_{j}$ of our previous guess~$\sigma^\tte_{j-1}$ as follows:
\begin{subequations}\label{sig-update-intro-fkC}
\begin{align}
\sigma^{\tte}_{{j}}&=\sigma_{\widehat\jmath},\quad\mbox{with}\quad\widehat\jmath=\min\argmin\limits_{1\le i\le N} \fkC(z_{\sigma_i}-z_{\sigma},u_{\sigma_i}-u_{\sigma}),
\intertext{where, for the comparison functional~$\fkC$, we shall take}
\fkC(\zeta,\eta)&\coloneqq\norm{\zeta}{L^2(I_j^\tau,\bbR^p)}^2+\norm{\eta}{L^2(I_j^\tau,\bbR^m)}^2.
\end{align}
\end{subequations}
Thus  we shall compare norms of the pair of differences~$(z_{\sigma_i}-
 z_{\sigma},u_{\sigma_i}-
 u_{\sigma})$.

The feedback operator~$K_{\sigma^\tte_{ {j}}}$ associated with~$\sigma^\tte_j=\sigma_{\widehat\jmath}\in\varSigma$ constructed in the interval~$I_\tau^j$ is used in the next interval of time~$I_\tau^{j+1}$. Repeating these steps we construct a piecewise constant update~$\sigma^\tte(t)\in\varSigma$ for the true~$\sigma$, with $\sigma^\tte(t)=\sigma^\tte_{ {j-1}}$ for~$t\in I_\tau^{j}$, $j\in\bbN_+$.

\begin{remark}
We focus on uncertainties in the modeling of the free dynamics only. We also include a numerical result on the response of the constructed input against uncertainties on the state~$y_{t_0}$ used in~\eqref{sys-Feed-intro-aux}, at the concatenating times~$t_0\in\tau\bbN$.
\end{remark}

\begin{definition}
Let~$\varepsilon>0$ be given. We say that~$\varSigma$ is~$\varepsilon$-dense in~$\fkS$ if for all~$\sigma\in\fkS$ there exists~$\varsigma\in\varSigma$ such that~$\norm{\varsigma-\sigma}{}\le\varepsilon$.
\end{definition}

We will see (in section~\ref{S:stab-detect}) that if~$\varSigma$ is~$\varepsilon$-dense in~$\fkS$, with~$\epsilon$ small enough, then for every $\sigma \in \fkS$ there is $\varsigma \in \varSigma$ such that \eqref{sys-Feed-intro} is stable, that is, there is a stored feedback~$K_\varsigma$ stabilizing the true system.

 \subsection{Contents and notation}
In Section~\ref{S:algorithm} we present an algorithm to construct a library of time-$\rho$ periodic input feedback operators  for each parameter in a finite training set~$\varSigma\subset\fkS$. Then, we present the algorithm for the online update of the estimate~$\sigma^\tte\in\varSigma$ of the uncertain parameter~$\sigma\in\fkS$, based on IO data comparison. In section~\ref{sS:robustness} we present  results, showing that, if the strategy provides us with an estimate~$\sigma^\tte(t)$ close enough to the unknown~$\sigma$, then the constructed feedback control input will be able to stabilize our system.
 A discussion on the ability of the strategy in providing us with an appropriate update is presented in section~\ref{sS:rmks-Conj}. Finally, results of several numerical examples are presented in section~\ref{S:numerics} showing the stabilizing performance of the input control provided by the proposed strategy.

\smallskip

Concerning notation, we denote by~$\bbR^{m\times n}$ the space of matrices with real entries, with~$m$ rows and~$n$ columns. We denote by~$\bbR^{1\times k}_{\rm inc}\subset\bbR^{1\times k}$ the subset of strictly increasingly ordered row vectors. For simplicity, for a given subset~$\clI\subset\bbR$ and row vector~$v\in\bbR^{1\times k}$, we shall write~$v\subset\clI$ if~$v_{1,j}\in \clI$ for all $1\le j\le k$. Finally, we denote~$\bbR^{1\times :}_{\rm inc}\coloneqq\bigcup_{k=2}^{+\infty}\bbR^{1\times k}_{\rm inc}$.

Further, let~$\clM(i_1:i_2,j_1:j_2)\in\bbR^{(i_2-i_1+1)\times(j_2-j_1+1)}$ denote the block composed of the entries~$\clM_{ij}$ (in $i$-th row and $j$-th column) in the matrix~$\clM$ with~$i_1\le i\le i_2$, and~$j_1\le j\le j_2$.
Finally, we denote by~$[a,b]_M\coloneqq\{a+\frac{k-1}{M-1}(b-a)\mid 1\le k\le M\}$ a regular mesh in the interval~$[a,b]\subset\bbR$ with~$M$ nodes, $a<b$, $M\ge2$.

\section{The algorithm}\label{S:algorithm}
We present an algorithm to compute a switching  adaptive stabilizing feedback control input based on an offline-online  strategy.
We are given a dynamical control system
\begin{equation}\label{sys-real}
\dot y_\sigma=\clA_\sigma y_\sigma +Bu,\quad y_\sigma(0)=y_\rmo,
\end{equation}
where, now, we write subscripts in the state to underline its dependence on the uncertain parameter~$\sigma\in\fkS$. We preset the offline and online components of the algorithm and show the stabilizing properties of the generated feedback control input, based on a piecewise constant update~$\sigma_\tte$ for~$\sigma$. This estimate is updated online at time instants~$k\tau$, for an a priori fixed~$\tau>0$, $k\in \bbN_+$.

The offline steps of the strategy are shown in Algorithm~\ref{Alg:Offl}, where we essentially store the Riccati feedbacks for a chosen set of training parameters~$\varSigma$.

Within the algorithms we assume that $B\in\bbR^{n\times m}$, $C\in\bbR^{p\times n}$ and~$A_\sigma\in\bbR^{n\times n}$ stand for matrices representing (e.g., after spatial discretization) our homonym operators~$B$, $C$, and~$A_\sigma$. Let us denote the set of matrices~$\clA_\varSigma\coloneqq\{\clA_{\sigma_i}\mid \sigma_i\in\varSigma\}$.

\begin{algorithm}[ht]
 \caption{Offline storage stage}
\begin{algorithmic}[1]\label{Alg:Offl}
\REQUIRE{$\varSigma=(\sigma_i)_{i= 1}^{N}\subseteq\fkS$, $B$, $Q$, and $\clA_\varSigma\coloneqq\{\clA_{\sigma_i}\mid \sigma_i\in\varSigma\}$ with time-$\rho(\sigma_i)$ periodic elements $\clA_{\sigma_i}$.}
\ENSURE{Library $K_\varSigma^\clT=\{K_{\sigma_i}^\clT\mid 1\le i\le N\}$ containing the periodic  input feedbacks~$K_{\sigma_i}(t)$ for~$t\in\clT_{\sigma_i}^K\subset[0,\rho(\sigma_i))\bigcap \bbR^{1\times :}_{\rm inc}$;
}
\FOR{$i=1:N$}
\STATE Set~$\overline\sigma=\sigma_i$;
\STATE  Solve for the time-$\rho(\overline\sigma)$ periodic solution~$\Pi_{\overline\sigma}$ of equation~\eqref{Ricc-sig-feed};
\STATE  Store the time mesh~$\clT_{\overline\sigma}^K$ and~$K_{\overline\sigma}^\clT(t)=-B^\top\Pi_{\overline\sigma}(t)\in\bbR^{m\times n}$,~$t\in\clT_{\overline\sigma}^K$;
\ENDFOR
 \end{algorithmic}
\end{algorithm}

The online stage is based on a concatenation over the time intervals~$I_j^\tau$ in~\eqref{timeint-Iktau}. This concatenation is described in Algorithm~\ref{Alg:Onl-concat}, and the procedure in each interval~$I_j^\tau$ is presented in Algorithm~\ref{Alg:Onl-1int}.

\begin{algorithm}[ht]
 \caption{Online stage for an interval with a constant guess/estimate~$\sigma^{\rm old}\in\varSigma$}
\begin{algorithmic}[1]\label{Alg:Onl-1int}
\REQUIRE{$\varSigma^{*}=(\sigma_i^*)_{i= 1}^{N^*}\subseteq\varSigma$,  $\clA_\varSigma\coloneqq\{\clA_{\sigma_i}\mid \sigma_i\in\varSigma\}$ with time-$\rho(\sigma_i)$ periodic elements, $B$, $C$, $\sigma^{\rm old}\in\varSigma^{*}$, $t_0\ge0$, $\tau>0$, integer~$\bft\ge2$, $y^{\rm old}\in\bbR^{n\times1}$, input feedback~$K_\varSigma^\clT$  library stored by Algorithm~\ref{Alg:Offl}.}
\ENSURE{state~$y^{\rm new}\in\bbR^{n\times1}$ at time~$t=t_0+\tau$;  update~$\sigma^{\rm new}\in\varSigma^{*}$;
}
\STATE Set~$K=K_{\sigma^{\rm old}}^\clT$;
\STATE Set mesh~$\bfT=[t_0,t_0+\tau]_\bft$;
\STATE Set~$U=0\in\bbR^{mN^*\times \bft}$ and~$Z=0\in\bbR^{pN^*\times \bft}$\label{Alg:Onl-1int-Par1}
\STATE\label{Alg:Onl.run-real}  \begin{enumerate}[leftmargin=1em,itemindent=.2em,label=\alph*.]
\item Let the true system~$(\clA_{\sigma},B)$ in~\eqref{sys-real} (cf.~\eqref{sys-Feed-intro}) operate, for time~$t\in[t_0,t_0+\tau]$, under the input feedback~$u_\sigma=K(t) y_\sigma(t)$;\label{Alg:Onl.run-real1}
\item Set~$y^{\rm new}=y(t_0+\tau)$;
\item Store input~$u_\sigma(1:m,1:\bft)\in\bbR^{m\times \bft}$ and output~$z_\sigma(1:p,1:\bft)\in\bbR^{p\times \bft}$,
for time~$t\in\bfT$.\label{Alg:Onl.run-real3}
\end{enumerate}
\setcounter{ALC@line}{3}
\STATE {\bf for } $i=1:N^*$ {\bf do} \label{Alg:Onl.run-auxil}
\begin{enumerate}[leftmargin=1em,itemindent=.2em,label=\Alph*.]
\item Solve the auxiliary systems~$(\clA_{\sigma_i^*},B)$ in~\eqref{sys-real} (cf.~\eqref{sys-Feed-intro-aux}) for each~$\sigma_i^*\in\varSigma^*$, with initial state~$y_{\sigma_i^*}(t_0)=y^{\rm old}$, with feedback input~$u_{\sigma_i^*}(t)=K(t) y_{\sigma_i^*}(t)$;\label{Alg:Onl.run-auxil1}
\item Store input~$U(1+m(i-1):mi,1:\bft)=u_{\sigma_i^*}(1:m,1:\bft)$,\\
and output~$Z(1+m(i-1):mi,1:\bft)=Cy_{\sigma_i^*}(1:m,1:\bft)$, for time~$t\in\bfT$;\label{Alg:Onl.run-auxil2}
\end{enumerate}
{\bf end for}
\STATE Set~$E=0\in\bbR^{N^*\times 1}$;\label{Alg:Onl-1int-Upsig1}
\FOR{$i=1:N^*$}
\STATE  Set~$D_u=U(1+m(i-1):mi,1:\bft)-u_{\sigma}(1:m,1:\bft)$\\
and~$D_z=Z(1+m(i-1):mi,1:\bft)-z_{\sigma}(1:m,1:\bft)$
\STATE Set $E(i,1)=\fkC(D_z,D_u)$, with~$\fkC$ as in~\eqref{sig-update-intro-fkC};
\ENDFOR
\STATE
Set~$\sigma^{\rm new}=\sigma^*_{\widehat\jmath}$, with $\widehat\jmath\coloneqq\min\argmin\limits_{1\le j\le N^*}E(j,1)$;\label{Alg:Onl-1int-Upsig5}
 \end{algorithmic}
\end{algorithm}

We can see that in Algorithm~\ref{Alg:Onl-1int}, the steps~\ref{Alg:Onl.run-auxil}:(\ref{Alg:Onl.run-auxil1}--\ref{Alg:Onl.run-auxil2}) can be  performed in parallel with, and independently from, steps~\ref{Alg:Onl.run-real}:(\ref{Alg:Onl.run-real1}--\ref{Alg:Onl.run-real3}). Then, the update of the parameter estimate is performed through the steps~\ref{Alg:Onl-1int-Upsig1}--\ref{Alg:Onl-1int-Upsig5}; which involve elementary operations and can be performed in short time.

Though the computations within steps~\ref{Alg:Onl.run-auxil}:(\ref{Alg:Onl.run-auxil1}--\ref{Alg:Onl.run-auxil2}), in Algorithm~\ref{Alg:Onl-1int}, can be  performed  while our system is running (in steps~\ref{Alg:Onl.run-real}:(\ref{Alg:Onl.run-real1}--\ref{Alg:Onl.run-real3})), we may need a large computational effort to solve the equation for all parameters in~$\varSigma$ if $N=\#\varSigma$ is large. In some cases, we might be able to realize that some parameters can be discarded. This is the reason we consider a subset~$\varSigma^*\subset\varSigma$ within Algorithm~\ref{Alg:Onl-1int}. Within Algorithm~\ref{Alg:Onl-concat} we take this subset as the union~$\varSigma^\rmg\cup\varSigma^\rml\subseteq\varSigma$. The subset~$\varSigma^\rml$ contains parameters in a  local neighborhood of our latest estimate. It can be seen as an attempt at locally improving the estimate of the unknown parameter, while the use of the subset~$\varSigma^\rmg$ serves the purpose of keeping a  global look at the parameters in~$\varSigma$.

\begin{algorithm}[ht]
 \caption{Online stage: concatenation/implementation, for~$t\in(0,T)\subseteq(0,+\infty)$}
\begin{algorithmic}[1]\label{Alg:Onl-concat}
\REQUIRE{$\varSigma=(\sigma_i)_{i= 1}^ N$, $N_\rmg\le N$,  $\gamma\ge0$, $\clA_\varSigma$ with time periodic elements, $B$, $C$, $\widehat\sigma\in\varSigma$, $\tau>0$, $y_{\rme0}\in\bbR^{n\times1}$, integer $\bft\ge2$, $T\in(0,+\infty)\cup\{+\infty\}$.}
\STATE Set $t_0=0$; $y^{\rm old}=y_{\rme0}$;
\WHILE{$t_0< T$}
\STATE Set a random sequence~$\clI^\rmg$ of $N_\rmg$ integers in~$[1,N]$; set~$\varSigma^\rmg\coloneqq\{\sigma_i\mid i\in\clI^\rmg\}$;\label{Alg:Onl-concat-Sig*}
\STATE Set~$\varSigma^\rml\coloneqq \{\sigma\in\varSigma\mid\norm{\sigma-\widehat\sigma}{}\le\gamma\}$;
\STATE Set  $\Sigma^{*}=\varSigma^\rmg\cup\varSigma^\rml$;
\STATE Let our system operate for time~$t\in[t_0,t_0+\tau]$ and use Algorithm~\ref{Alg:Onl-1int} to find~$(y^{\rm new},\sigma^{\rm new})$;
\STATE Update~$y^{\rm old}=y^{\rm new}$; $\widehat\sigma=\sigma^{\rm new}$;
\STATE Shift $t_0\to t_0+\tau$;\label{Alg:Onl-concat-shiftt0}
\ENDWHILE
 \end{algorithmic}
\end{algorithm}

\begin{remark}
The stabilization problem  corresponds to taking~$T=+\infty$ in Algorithm~\ref{Alg:Onl-concat}. We include the case of a finite time-horizon, simply because it is convenient for later reference, when we apply the strategy up to a given finite time~$T>0$.
\end{remark}
\begin{remark}
In Algorithm~\ref{Alg:Offl} the time-$\rho$ periodic input feedback is saved for time~$t$ in a discrete mesh~$\clT(\sigma)\subset[0,\rho(\sigma))$. In Algorithm~\ref{Alg:Onl-1int}, we can apply these operators to a general~$t>0$ following, for example, the procedure in~\cite[sect.~3.6.2]{Rod23-eect}, namely, first we select~$\overline t\in[0,\rho)$ so that~$t-\overline t\in \rho\bbN$, next we find elements $\clT^{\rm ex}_{1,k}$ and~$\clT^{\rm ex}_{1,k+1}$ in~$\clT^{\rm ex}\coloneqq \begin{bmatrix}\clT &\rho\end{bmatrix}$ so that $\clT^{\rm ex}_{1,k}\le\overline t<\clT^{\rm ex}_{1,k+1}$, and use a convex combination of the stored operators at time instants~$\clT^{\rm ex}_{1,k}$ and~$\clT^{\rm ex}_{1,k+1}$.
\end{remark}

\section{Adaptive stabilization}\label{S:stab-detect}
The stabilizing performance of the constructed input depends on appropriate continuity properties. For this purpose, it is assumed  that the Riccati based feedback input operator~$K_\sigma$ depends continuously on the parameter~$\sigma\in\fkS$.  In addition, we need a set of training parameters~$\varSigma$ which is ~$\varepsilon$-dense in~$\fkS$, with~$\varepsilon>0$ sufficiently small.

\subsection{Stabilizability and detectability}\label{sS:StabDete}
Here, we recall some terminology for a fixed parameter $\sigma \in \fkS$. We start by considering~\eqref{sys-intro} under a feedback control input~$u(t)=K(t)y(t)$ as,
\begin{equation}\label{sys-free}
\dot y=(\clA_\sigma +BK)y,\quad y(0)=y_0,
\end{equation}
with~$\clA_\sigma\in L^\infty(\bbR_+;\clL(V,V'))$, control operator~$B\in\clL(\bbR^m,V')$, and input feedback  control operator~$K\in L^\infty(\bbR_+;\clL(V,\bbR^m))$.

\begin{definition}\label{D:stable}
Let~$\zeta\ge1$ and~$\mu>0$. The operator~$\clA_\sigma +BK$ is $(\zeta,\mu)$-stable if, for every~$y_0\in H$, the solution of~\eqref{sys-free} satisfies
\begin{equation}\notag
\norm{y(t)}{H}\le \zeta\rme^{-\mu(t-s)}\norm{y(s)}{H},\quad\mbox{for all}\quad t\ge s\ge0.
\end{equation}
\end{definition}

\begin{definition}\label{D:stabiliz}
The pair~$(\clA_\sigma,B)$ is $(\zeta,\mu)$-stabilizable if there is~$K\in L^\infty(\bbR_+;\clL(V,\bbR^m))$ such that~$\clA_\sigma+BK$ is $(\zeta,\mu)$-stable.
\end{definition}

Let us now consider also the operator~$Q\in\clL(V,H_Q)$ in the Riccati equation~\eqref{Ricc-sig-feed}.

\begin{definition}\label{D:detect}
The pair~$(\clA_\sigma,Q)$ is $(\zeta,\mu)$-detectable if there is~$L\in L^\infty(\bbR_+;\clL(H_Q,V'))$ such that~$\clA_\sigma+LQ$ is $(\zeta,\mu)$-stable.
\end{definition}

If there exists a pair~$(\zeta,\mu)$ as in Definition~\ref{D:stable}/\ref{D:stabiliz}/\ref{D:detect}, but we do not need to precise it, we will simply say  exponentially  stable/stabilizable/detectable, respectively.

\subsection{Towards adaptive stabilization}\label{sS:robustness}
We recall stability results under Riccati based feedback controls. Then, we address adaptive stabilization   properties.

\begin{assumption}
The uncertain parameter~$\sigma$ is contained in a compact subset~$\fkS\subset\bbR^s$.
\end{assumption}

\begin{assumption}\label{A:stabidetect-sig}
The pair~$(\clA_\sigma,B)$ is exponentially stabilizable and the pair~$(\clA_\sigma,Q)$ is exponentially detectable.
\end{assumption}

\begin{assumption}\label{A:cont-Asig}
The operator~$\clA_\sigma\in L^\infty(\bbR_+;\clL(V,V'))$ is time-$\rho$ periodic and depends continuously on~$\sigma\in\fkS$.
\end{assumption}

\begin{theorem}\label{T:stabRicc}
Under Assumptions~\ref{A:stabidetect-sig} and~\ref{A:cont-Asig}, if~$\Pi_\sigma\succeq0$ solves~\eqref{Ricc-sig-feed}, then~$\clA_\sigma-BB^\top\Pi_\sigma$ is exponentially stable.
\end{theorem}

Concerning Theorem~\ref{T:stabRicc}, we know that once we have stabilizability (cf. Assumption~\ref{A:stabidetect-sig}), the solution minimizing a  suitable linear quadratic problem is given by a positive semi-definite solution of~\eqref{Ricc-sig-feed} (cf.~\cite[sect~2.3]{Rod23-eect}). Uniqueness follows from detectability in Assumption~\ref{A:stabidetect-sig} (see~\cite[Thm.~3.8]{DaPrato87}). The time-$\rho$ periodicity of~$\Pi_\sigma$ follows from the time periodicity in Assumption~\ref{A:cont-Asig} together with the dynamic programming principle, recalling that~$\frac12(\Pi_\sigma(s) w,w)_H$ will give us the ``cost to go'' from time $t=s$ up to~$t=+\infty$ for a given state~$y(s)=w$. Indeed, the periodicity of the dynamics imply that we must have~$\frac12(\Pi_\sigma(s) w,w)_H=\frac12(\Pi_\sigma(s+\rho) w,w)_H$ (alternatively, we can follow  arguments as in~\cite[Proof of Prop.~3.4]{DaPratoIchikawa90}).

Next, we present continuity results towards robustness against small perturbations of~$\sigma$. We need an assumption on the continuous dependence of the Riccati solutions on~$\sigma$;  see~\cite[sect.~4]{RanRodman88} for finite-dimensional systems and~\cite[Thm.~V.2]{Curtain13} including infinite-dimensional systems, for results in the autonomous case.
\begin{assumption}\label{A:cont-Pisig}
The solution~$\Pi_\sigma\in L^\infty((0,\rho);\bbR^{n\times n}_{\succeq0})$ of the Riccati equation~\eqref{Ricc-sig-feed} depends continuously on~$\sigma\in\fkS$. Furthermore, there exist~$D\ge1$ and~$\lambda>0$ such that for all~$\sigma\in\fkS$ we have that~$\clA_\sigma-BB^\top\Pi_{\sigma}$  is~$(D,\lambda)$-stable.\end{assumption}

\begin{lemma}\label{L:stabfeed-pertsig}
Under Assumption~\ref{A:cont-Pisig}, there exists~$\varepsilon>0$ such that for every~$(\sigma,\varsigma)\in\fkS\times\fkS$ with~$\norm{\sigma-\varsigma}{}<\varepsilon$, it follows that~$\clA_\sigma-BB^\top\Pi_{\varsigma}$ is stable.
\end{lemma}

\begin{proof}
We write~$\clA_\sigma-BB^\top\Pi_{\varsigma}=\clA_\varsigma-BB^\top\Pi_{\varsigma}+\clA_\sigma-\clA_\varsigma$. By Assumption~\ref{A:cont-Pisig}, the matrix~$A\coloneqq\clA_\varsigma-BB^\top\Pi_{\varsigma}$ is $(D,\lambda)$-stable, with~$D\ge1$ and~$\lambda>0$. Due to Duhamel's formula,
every solution~$y$ of
\begin{align}\notag
\dot y=(\clA_\sigma-BB^\top\Pi_{\varsigma})y
\end{align}
satisfies, for~$t\ge s_0\ge0$,
\begin{align}\notag
\norm{y(t)}{}\le D\rme^{-\lambda (t-s_0)}\norm{y(s_0)}{}+D\int_{s_0}^t\rme^{-\lambda(t-s)}\norm{(\clA_\sigma-\clA_\varsigma) y(s)}{}\,\rmd s,
\end{align}
which leads us to
\begin{align}
&\norm{y(t)}{}^2\le 2D^2\left(\rme^{-2\lambda (t-s_0)}\norm{y({s_0})}{}^2+\left(\int_{s_0}^t\rme^{-\lambda(t-s)}\norm{(\clA_\sigma-\clA_\varsigma) y(s)}{}\,\rmd s\right)^2\right).\label{2sig-estL2y-1}
\end{align}
With~$E\coloneqq\clA_\sigma-\clA_\varsigma$, and~$\|E\|\coloneqq\norm{E}{L^\infty((0,+\infty);\clL(H))}$, the integration of the last term over time~$t\in({s_0},+\infty)$ leads  to
\begin{align}
&\int_{s_0}^{+\infty}\left(\int_{s_0}^t\rme^{-\lambda(t-s)}\norm{E y(s)}{}\,\rmd s\right)^2\rmd t\notag\\
\le&\int_{s_0}^{+\infty}\left(\int_0^t\rme^{-\lambda(t-s)}\,\rmd s\right)\left(\int_{s_0}^t\rme^{-\lambda(t-s)}\norm{E y(s)}{}^2\,\rmd s\right)\rmd t\notag\\
\le&\frac{\dnorm{E}{}^2}{\lambda}\int_{s_0}^{+\infty}\rmd t\int_{s_0}^t\rme^{-\lambda(t-s)}\norm{y(s)}{}^2\,\rmd s=\frac{\dnorm{E}{}^2}{\lambda}\int_{s_0}^{+\infty}\norm{y(s)}{}^2\rmd s\int_s^{+\infty}\rme^{-\lambda(t-s)}\,\rmd t,\notag
\end{align}
which, together with~\eqref{2sig-estL2y-1}, implies
\begin{align}\notag
&\int_{s_0}^{+\infty}\norm{y(t)}{}^2\,\rmd t\le 2D^2\left(\frac{1}{2\lambda}\norm{y({s_0})}{}^2+\frac{\dnorm{E}{}^2}{\lambda^2}\int_{s_0}^{+\infty}\norm{y(s)}{}^2\rmd s\right).
\end{align}
By Assumption~\ref{A:cont-Asig}, we  can choose~$\varepsilon>0$ small enough so that~$\norm{\sigma-\varsigma}{}<\varepsilon$ implies~$\dnorm{E}{}^2<(2D^2)^{-1}\lambda^2$. This leads to
\begin{align}\label{2sig-estL2y-2}
&\int_{s_0}^{+\infty}\norm{y(t)}{}^2\,\rmd t\le \frac{1}{1-2D^2\frac{\dnorm{E}{}^2}{\lambda^2}}\frac{2D^2}{2\lambda}\norm{y({s_0})}{}^2= \frac{\lambda D^2}{\lambda^2-2D^2\dnorm{E}{}^2}\norm{y({s_0})}{}^2.
\end{align}
The exponential stability of~$\clA_\sigma-BB^\top\Pi_{\varsigma}$ follows by Datko's Theorem~\cite[Thm.~1]{Datko72}.
\end{proof}

Furthermore, following the proof in~\cite[Thm.~1]{Datko72}, we see that the analogue to~\eqref{2sig-estL2y-2} corresponds to~\cite[Equ.~(7)]{Datko72}, this leads us to the analogue of~\cite[Equ.~(16)]{Datko72} as
\begin{align}\label{2sig-estL2y-3}
&\norm{y(t)}{}\le C_0\overline D\mu\norm{y({s_0})}{},\qquad t\ge s_0+\mu^{-2},
\end{align}
for some constants~$C_0$ and~$\mu$ such that~$\overline D\mu<1$, where~$\overline D\coloneqq(\frac{2D^2}{\lambda^2-2D^2\dnorm{E}{}^2})^\frac12$, with~$D$ as in~\eqref{2sig-estL2y-2}. Next, following the proof in~\cite[Lem.~1]{Datko72} we can conclude that we can achieve any exponential decrease as~$\mu^2$, thus  satisfying~$\mu^2<\frac{\lambda^2-2D^2\dnorm{E}{}^2}{\lambda D^2}$. That is, for smaller~$\dnorm{E}{}$ we  can guarantee
 decrease rates closer to~$\frac{\lambda}{D}$, where~$(D,\lambda)$ is the stability  guaranteed by the Riccati feedback input operators (see Assum.~\ref{A:cont-Pisig} and Def.~\ref{D:stable}).

\subsection{On the success of the strategy}\label{sS:rmks-Conj}
 Lemma~\ref{L:stabfeed-pertsig} expresses the fact that, if we manage to find an estimate/update~$\varsigma\in\varSigma$ close enough to the unknown parameter~$\sigma$, then we can use the stored feedback~$ K_\varsigma=-B^\top\Pi_\varsigma$ to construct an input stabilizing control. At this point, we would like to know if the strategy we propose will be able to find such a~$\varsigma$. The short answer is that, in general, it will not. Let us assume, for simplicity, that we solve the auxiliary systems for all~$\varsigma\in\varSigma$. If~$\varSigma$ is $\varepsilon$-dense in~$\fkS$, then there will be parameters in~$\varSigma$, close to~$\sigma$, for which the IO data is close to the true one, but this does not imply that there is no parameter farther from~$\sigma$ with closer  IO data. However, we underline that we do not need a precise estimate of~$\sigma$. The knowledge of the dynamics defined by the operator~$\clA_\sigma$ is sufficient for the strategy to provide us with a stabilizing feedback.

The use of finite-dimensional  IO data proposed in Algorithm~\ref{Alg:Onl-1int} is a natural choice, since such data is at our disposal (assuming that the output~$z=Cy$ is available as the result of sensor measurements), in particular, the data comparison is not numerically expensive and can be done in real time.

We do not have a complete characterization of the systems for which the IO data comparison is enough to guarantee the success of the strategy in providing us with a stabilizing feedback input (as the simulations do suggest). But, let us present examples illustrating some mechanisms of the strategy. In these examples we consider, for simplicity, a finite dimensional system and assume that the updating time-horizon~$\tau$ is small so that
an explicit Euler finite-difference scheme with temporal step~$\tau$ gives us a good approximation of the dynamics of the controlled system (or at least of the inputs and outputs). Then, a smaller temporal step~$\xi=\frac\tau S$, ~$S\in\bbN$ will also give us a good approximation. For simplicity (without loss of generality) let us restrict ourselves to the first updating time interval~$I_1^\tau=(0,\tau)$ (an analogue argument can be used for the translated intervals~$I_j^\tau=(j-1)\tau+I_1^\tau$). Denoting by~$y^j$ the state at time instant~$t=j\xi$, $1\le j\le S$ we would obtain (approximately)
\begin{align}
y_\sigma^j=y_\sigma^{j-1}+\xi\clA_\sigma^K((j-1)\xi) y_\sigma^{j-1}=\left(\Id+\xi\clA_\sigma^K((j-1)\xi)\right) y_\sigma^{j-1},\qquad y_\sigma^0=y_0,
\end{align}
in the time interval~$I_1^\tau$. Thus
\begin{align}\label{Exaut1-fd}
y_\sigma^j&=\Xi^{[j-1]}_\sigma\cdots\Xi^{[1]}_\sigma\Xi^{[0]}_{\sigma}y_0,\quad\mbox{where}\quad\Xi^{[i]}_\sigma\coloneqq\left(\Id+\xi\clA_\sigma^K(i\xi)\right)
\end{align}
 and~$\clA_\sigma^K(t)\coloneqq \clA_\sigma(t)+BK(t)$ is the operator defining the controlled dynamics. In this case the input and output are, at time instants~$t=j\xi$,
given by~$u_\sigma^j=Ky_\sigma^j$ and~$z_\sigma^j=Cy_\sigma^j$.

Assume that~$\sigma$ is the true parameter and that~$\varsigma\in\varSigma^*\subseteq\varSigma$ is a parameter in the training subset, giving us the corresponding inputs~$u_\varsigma^j$ and outputs~$z_\varsigma^j$. For the differences~$\delta_u^j\coloneqq u_\varsigma^j- u_\sigma^j$ and~$\delta_z^j\coloneqq z_\varsigma^j- z_\sigma^j$ we find
\begin{align}\label{diffIO.Ex}
\delta_u^j&=K\delta_\clA^jy_0,\quad
\delta_z^j=C\delta_\clA^jy_0,\quad\mbox{with} \quad \delta_\clA^j\coloneqq \Xi^{[j-1]}_\varsigma\cdots\Xi^{[1]}_\varsigma\Xi^{[0]}_{\varsigma}-\Xi^{[j-1]}_\sigma\cdots\Xi^{[1]}_\sigma\Xi^{[0]}_{\sigma}.
\end{align}
We would like to know whether the smallness of the~$\norm{\delta_u^j}{\bbR}$ and~$\norm{\delta_z^j}{\bbR}$ imply that~$\clA_\varsigma$ is close to~$\clA_\sigma$. We address this question within the following examples.

\begin{example}\label{Ex:aut1}
Let us consider system~\eqref{sys-intro} with matrices
\begin{equation}\notag
\clA_\sigma=\begin{bmatrix}0&\sigma\\0&0\end{bmatrix},\qquad B=\begin{bmatrix}0 &1\end{bmatrix}^\top,\qquad C=\begin{bmatrix}1 &0\end{bmatrix},
\end{equation}
 and uncertain~$\sigma\in[1,2]$. The pair~$(A,B)$ is controllable (accordingly to the Kalman rank condition~\cite[Part~I, Sect.~1.3, Thm.~1.2]{Zabczyk08}, thus also stabilizable~\cite[Part~I, Sect.~2.5, Thm.~2.9]{Zabczyk08}). The free dynamics is unstable for~$\sigma>0$.
Given an input stabilizing operator~$K=\begin{bmatrix}\kappa_{1}&\kappa_{2}\end{bmatrix}$, we find
\begin{equation}\notag
A_\sigma^K=\begin{bmatrix}0&\sigma\\\kappa_1&\kappa_2\end{bmatrix}.
\end{equation}

Now, in order to use~\eqref{Exaut1-fd}, with $\xi=\frac\tau2$, we find that~$\Xi^{[1]}_\sigma=\Xi^{[0]}_\sigma$ and, explicitly,
\begin{align}\notag
\Xi^{[0]}_{\sigma}=\begin{bmatrix}1&\xi\sigma\\ \xi\kappa_1&1+\xi\kappa_2\end{bmatrix},\quad
\Xi^{[1]}_{\sigma}\Xi^{[0]}_{\sigma}=\begin{bmatrix}1+\xi^2\sigma\kappa_1&\xi\sigma(2+\xi\kappa_2)\\ \xi\kappa_1(2+\xi\kappa_2)&\xi^2\sigma\kappa_1+(1+\kappa_2)^2\end{bmatrix}.
\end{align}
Using~\eqref{diffIO.Ex}, we find the output and input differences, at time instants~$t\in\{\frac{\tau}2,\tau\}$,
\begin{align}\notag
&\delta_z^1=(\varsigma-\sigma)\begin{bmatrix}0&\xi\end{bmatrix}y_0 &&\mbox{and} \qquad \delta_z^2=(\varsigma-\sigma)\begin{bmatrix}\xi^2\kappa_1&\xi(2+\xi\kappa_2)\end{bmatrix}y_0. \notag\\
&\delta_u^1=(\varsigma-\sigma)\begin{bmatrix}0&\xi\kappa_1\end{bmatrix}y_0 &&\mbox{and} \qquad \delta_u^2=(\varsigma-\sigma)\begin{bmatrix}\xi^2\kappa_1^2&\xi(2+\xi\kappa_2))\kappa_1\end{bmatrix}y_0. \notag
\end{align}

In case~$y_0=0$, we will have that the IO data will vanish, and so every feedback input will keep the state at rest. Thus, we consider only the nontrivial case~$y_0\eqqcolon(y_{01},y_{02})\ne(0,0)$. If~$y_{02}\ne0$ then~$
\norm{\delta_z^1}{\bbR}$ is small only if $\norm{\varsigma-\sigma}{\bbR}$ is. Since~$K$ is stabilizing, we necessarily have that~$\kappa_1<0$, thus we must have that~$\norm{\varsigma-\sigma}{\bbR}$ is small. Thus, in this case~$\norm{\delta_u^1}{\bbR}$ is small only if $\norm{\varsigma-\sigma}{\bbR}$ is.
 If~$y_{02}=0$, we cannot draw conclusions from~$\norm{\delta_u^1}{\bbR}$ or~$\norm{\delta_z^1}{\bbR}$, but then~$\norm{\delta_z^2}{\bbR}$ and~$\norm{\delta_u^2}{\bbR}$ are both small only if~$\norm{\varsigma-\sigma}{\bbR}$ is small. In summary we have $\sum_{i=1}^2 \norm{\delta_z^i}{\bbR}^2 + \sum_{i=1}^2 \norm{\delta_u^i}{\bbR}^2$ is small only if~$\norm{\varsigma-\sigma}{\bbR}$ is small. Hence, we can expect the proposed strategy to provide us with updates close to the true parameter~$\sigma$.
\end{example}

\begin{example}\label{Ex:aut2}
Next, we consider system~\eqref{sys-intro} with matrices
\begin{equation}\notag
\clA_\sigma=\begin{bmatrix}0&1\\-1&\sigma\end{bmatrix},\qquad B=\begin{bmatrix}0 &1\end{bmatrix}^\top,\qquad C=\begin{bmatrix}1 &0\end{bmatrix},
\end{equation}
 and uncertain~$\sigma\in[0,1]$.  The pair~$(A,B)$ is controllable. The free dynamics is unstable for~$\sigma>0$. Using again~\eqref{Exaut1-fd}, with $\xi=\frac\tau2$, for a stabilizing feedback~$K=\begin{bmatrix}\kappa_{1}&\kappa_{2}\end{bmatrix}$, we find~$\Xi^{[1]}_\sigma=\Xi^{[0]}_\sigma$ and
\begin{align}\notag
\Xi^{[0]}_\sigma&=\begin{bmatrix}1&\xi\\\xi(\kappa_1-1)&1+\xi(\kappa_2+\sigma)\end{bmatrix},\notag\\
\Xi^{[1]}_\sigma\Xi^{[0]}_\sigma&=\begin{bmatrix}1+\xi^2(\kappa_1-1)&\xi(2+\xi(\kappa_2+\sigma))\\\xi(\kappa_1-1)(2+\xi(\kappa_2+\sigma))&\xi^2(\kappa_1-1)+(1+\xi(\kappa_2+\sigma))^2
\end{bmatrix}.\notag
\end{align}
Using~\eqref{diffIO.Ex}, we find the input and output differences, at time instants~$t\in\{\frac{\tau}2,\tau\}$,
\begin{align}\notag
\delta_u^1&=(\varsigma-\sigma)\begin{bmatrix}0&\xi\kappa_2\end{bmatrix}y_0,\notag\\ \delta_u^2&=\begin{bmatrix}(\varsigma-\sigma)\xi^2(\kappa_1-1)\kappa_2&(\varsigma-\sigma)\xi^2\kappa_1+((1+\xi(\kappa_2+\varsigma))^2-(1+\xi(\kappa_2+\sigma))^2)\kappa_2\end{bmatrix}y_0\notag\\
&=(\varsigma-\sigma)\begin{bmatrix}\xi^2(\kappa_1-1)\kappa_2&\xi^2\kappa_1+2\xi \kappa_2(1+\kappa_2\xi)+\xi^2\kappa_2(\varsigma+\sigma)\end{bmatrix}y_0.\notag\\
\delta_z^2&=(\varsigma-\sigma)\begin{bmatrix}0&\xi^2\end{bmatrix}y_0,\notag
\end{align}
Here, $\delta_z^1=0$ does not depend on the parameter $\sigma$.
Again, let us consider the nontrivial case~$(y_{01},y_{02})\ne(0,0)$. Since~$K$ is stabilizing, we necessarily have that~$\kappa_1-1<0$ and~$\kappa_2+\sigma<0$, in particular,~$\kappa_2<-\sigma\le0$.

 If~$y_{02}\ne0$, a small~$\norm{\delta_u^1}{\bbR}$ implies that~$\norm{\varsigma-\sigma}{\bbR}$ is also small. It follows that~$\norm{\delta_z^1}{\bbR}^2 + \norm{\delta_u^1}{\bbR}^2$ is small only if~$\norm{\varsigma-\sigma}{\bbR}$ is small. If~$y_{02}=0$, we see that~$\norm{\delta_u^1}{\bbR}=\norm{\delta_z^1}{\bbR}=\norm{\delta_z^2}{\bbR}=0$, and that~$\sum_{i=1}^2 \norm{\delta_z^i}{\bbR}^2 + \norm{\delta_u^i}{\bbR}^2$ is small only if~$\norm{\delta_u^2}{\bbR}^2$  is small, that is, only if~$\norm{\varsigma-\sigma}{\bbR}$ is small. Again,  we can expect the strategy to provide us with updates close to the true parameter~$\sigma$ (as confirmed, in section~\ref{sS:num-osc}, by results of simulations).
\end{example}

\begin{example}\label{Ex:aut3}
We consider system~\eqref{sys-intro} with matrices
\begin{equation}\notag
\clA_\sigma=\begin{bmatrix}1&1\\0&\sigma\end{bmatrix},\qquad B=\begin{bmatrix}1&0\end{bmatrix}^\top,\qquad C=\begin{bmatrix}1 &0\end{bmatrix},
\end{equation}
 and uncertain~$\sigma\in[-2,-1]$.  The pair~$(A,B)$ is controllable. The free dynamics is unstable. For a stabilizing feedback~$K=\begin{bmatrix}\kappa_{1}&\kappa_{2}\end{bmatrix}$, we find~$\Xi^{[j-1]}=\Xi^{[0]}_\sigma$, $j\in\bbN_+$, and
\begin{align}\notag
\Xi^{[j-1]}_\sigma\cdots\Xi^{[1]}_\sigma\Xi^{[0]}_\sigma=(\Xi^{[0]}_\sigma)^j=\begin{bmatrix}(1+\xi (1+\kappa_1))^j&  \frac{(1+\xi (1+\kappa_1))^j-(1+\xi \sigma)^j}{\xi (1+\kappa_1)-\xi\sigma}(1+\xi \kappa_2)\\0&(1+\xi \sigma)^j
\end{bmatrix}.\notag
\end{align}
and
\begin{align}\notag
\delta_\clA^j=(\Xi^{[0]}_\varsigma)^j-(\Xi^{[0]}_\sigma)^j =\begin{bmatrix}0& \left( \frac{(1+\xi (1+\kappa_1))^j-(1+\xi \varsigma)^j}{\xi (1+\kappa_1)-\xi\varsigma} - \frac{(1+\xi (1+\kappa_1))^j-(1+\xi \sigma)^j}{\xi (1+\kappa_1)-\xi\sigma} \right)(1+\xi \kappa_2)\\0&(1+\xi \varsigma)^j-(1+\xi\sigma)^j
\end{bmatrix}.\notag
\end{align}
We can see that if the initial state satisfies~$y_{02}=0$, then~$\delta_u^j=\delta_z^j=0$, thus we obtain no information on~$\varsigma-\sigma$, in fact, the auxiliary fictitious IO data coincides with the true one. In particular, every parameter will correspond to a minimizer of the IO data difference. The strategy will still give us an update value, which is not necessarily close to~$\sigma$. However, note that the provided input will still be stabilizing. Indeed, the solution~$y=(y_1,y_{2})$ will satisfy~$y_2(t)=\rme^{\sigma t}y_{02}$ independently of~$K$. Then~$y_1(t)=\rme^{(1+\kappa_1)t}y_{01}+\int_0^t\rme^{(1+\kappa_1)(t-r)}(1+\kappa_2)\rme^{\sigma r}y_{02}\,\rmd r$. For a stabilizing feedback~$K$ we necessarily have~$1+\kappa_1<0$, hence~$y(t)$ will  converge exponentially to zero with rate~$\mu=\min\{-(1+\kappa_1),-\sigma\}>0$ (cf.~\cite[Prop.~3.2]{AzmiRod20}).

Note that, this example has the particularity that the parameter~$\sigma<0$ does not contribute as a source of instability, but rather as a source of stability, this is why its knowledge is not essential for the design of stabilizing feedback control inputs (however, note that we cannot obtain an  exponential decrease rate~$\mu$ larger than~$-\sigma$).
\end{example}

\begin{example}\label{Ex:per1}
Finally, we consider system~\eqref{sys-intro} with
\begin{align}
\clA_{\sigma}(t)=\Psi(\rho^{-1}t +\phi)\begin{bmatrix}0 &1\\1&0\end{bmatrix},\qquad B=\begin{bmatrix}0 &1\end{bmatrix}^\top,\qquad C=\begin{bmatrix}1 &0\end{bmatrix},
\end{align}
where~$\Psi(t)=\Psi(t+1)\in\bbR$ is a time-$1$ periodic function with isolated zeros. The pair~$(\clA_{\sigma},B)$ is controllable (accordingly to the nonautonomous Kalman rank condition~\cite[Thm.~3]{SilvermanMeadows67}).
The free dynamics is not asymptotically stable if~$\vartheta\coloneqq\int_0^\rho\Psi(\rho^{-1}t +\phi)\,\rmd t\ge0$. Indeed, the solution~$y=(y_1,y_2)$ of~\eqref{sys-intro} with $u=0$, issued from the initial state~$y(0)=(1,1)$, satisfies
\begin{equation}\notag
y_1(t)=y_2(t)=\rme^{\int_0^t\Psi(\rho^{-1}s +\phi)\,\rmd s}
\end{equation}
and for~$t=j\rho$, $j\in\bbN$, we find
\begin{equation}\notag
y_1(j\rho)=y_2(j\rho)=\rme^{\int_0^{j\rho}\Psi(\rho^{-1}s +\phi)\,\rmd s}=\rme^{j\vartheta}.
\end{equation}
In particular, we will have that~$y_1(t)$ converges to zero as ~$t\to+\infty$ only if~$\vartheta<0$.

Let us  compute
\begin{align}\notag
\Xi_\sigma^{[0]}&=\begin{bmatrix}1 &\xi\Psi(\phi)\\\xi(\Psi(\phi)+\kappa_1(0))&1+\xi\kappa_2(0)\end{bmatrix},\notag\\
\Xi_\sigma^{[1]}&=\begin{bmatrix}1 &\xi\Psi(\rho^{-1}\xi+\phi)\\\xi(\Psi(\rho^{-1}\xi+\phi)+\kappa_1(\xi))&1+\xi\kappa_2(\xi)\end{bmatrix}.\notag
\end{align}
At time~$t=\xi$, for another parameter~$\varsigma=(\rho_*,\phi_*)$ we find the output and input
differences
\begin{align}\notag
\delta_z^1&= C(\Xi_\varsigma^{[0]}-\Xi_\sigma^{[0]})y_{0},\qquad
 \delta_u^1=K(\Xi_\varsigma^{[0]}-\Xi_\sigma^{[0]})y_{0}.
 \end{align}
We see that in case~$\Psi(\phi_*)=\Psi(\phi)$, then~$\Xi_\varsigma^{[0]}=\Xi_\sigma^{[0]}$ and~$\delta_z^1=0=\delta_u^1$. Note that~$\Psi(\phi_*)=\Psi(\phi)$ does not necessarily imply that~$\phi_*$ is close to~$\phi$. Further, in this case we also find
\begin{align}\notag
&\Xi_\varsigma^{[1]}\Xi_\varsigma^{[0]}-\Xi_\sigma^{[1]}\Xi_\sigma^{[0]}=(\Xi_\varsigma^{[1]}-\Xi_\sigma^{[1]})\Xi_\sigma^{[0]}\notag\\
&\hspace{2em}=\begin{bmatrix}0&\xi(\Psi(\rho_*^{-1}\xi+\phi_*)-\Psi(\rho^{-1}\xi+\phi))\\\xi(\Psi(\rho_*^{-1}\xi+\phi_*)-\Psi(\rho^{-1}\xi+\phi))&0
\end{bmatrix}\Xi_\sigma^{[0]}.\notag
\end{align}
At time~$t=2\xi$,  we find the output and input differences
\begin{align}\notag
\delta_z^2&=\xi\delta_\Psi\begin{bmatrix}0&1\end{bmatrix}\Xi_\sigma^{[0]}y_0,\qquad
\delta_u^2=\xi\delta_\Psi\begin{bmatrix}\kappa_2(2\xi)& \kappa_1(2\xi)\end{bmatrix}\Xi_\sigma^{[0]}y_0
\end{align}
with~$\delta_\Psi\coloneqq\Psi(\rho_*^{-1}\xi+\phi_*)-\Psi(\rho^{-1}\xi+\phi)$. Again if~$\delta_\Psi=0$ we will have~$\delta_z^2=0=\delta_u^2$, but not necessarily that~$\rho^{-1}\xi+\phi$ is close to~$\rho_*^{-1}\xi+\phi_*$. Note that if~$\Psi(\phi_*)=\Psi(\phi)$ and~$\xi$ is small we may have that~$\delta_\Psi\ne0$, but still rather small, which lead to small~$\delta_z^2$ and~$\delta_u^2$.

In conclusion, we can guess that, for such uncertainties there may exist parameters far from the true one corresponding to close IO data.
So, we can hardly expect that the strategy will provide us with the parameter closest to the true one (at least, if the true one is not in the training set). Note, however, that to have~$\delta_z^k=0=\delta_u^k$ at time~$t=k\xi$, $0\le k\le S-1$, we  need the conditions~$\Psi(\rho_*^{-1}k\xi+\phi_*)=\Psi(\rho^{-1}k\xi+\phi)$. Then for~$S$ large enough and~$\xi=\frac{\tau}S$, we can expect the  IO data comparison to provide us with a parameter for which the functions~$t\mapsto\Psi(\rho_*^{-1}t+\phi_*)$ and~$t\mapsto\Psi(\rho^{-1}t+\phi)$ are close to each other, for~$t\in I_1^\tau= (0,\tau)$ and, consequently,  an approximation of the matrix~$\clA_\sigma$ defining the (free) dynamics can be found.
In section~\ref{sS:num-per} we shall present simulations, concerning this example.

 Finally, recall that, though the pair~$(\clA_{\sigma},B)$ is controllable in nonempty open time intervals~\cite[Thm.~3]{SilvermanMeadows67}, we may loose stabilizability if we freeze time, namely, the pair~$(\clA_{\sigma}(s),B)$ is not stabilizable if~$\Psi(\rho^{-1}s +\phi)=0$; further,~$(\clA_{\sigma}(s),C)$ is not detectable, though~$(\clA_{\sigma},B)$ is observable in nonempty open time intervals~\cite[Thm.~5]{SilvermanMeadows67}.

\end{example}
\begin{remark}\label{R:rmks-gen}
After spatial discretization, the above arguments can be used in examples involving infinite-dimensional systems. In  section~\ref{sS:num-PDE} we present simulations involving a parabolic dynamics with an uncertainty in the convection term.
\end{remark}

\section{Numerical results}\label{S:numerics}
We present results of numerical experiments showing the viability of the approach. By comparing the true IO data to the fictitious one, we will be able to obtain a piecewise constant parameter update leading  to a  stabilizing input control, based on (time-periodic) Riccati feedbacks stored in a library.
The simulations have been run with Matlab.

\subsection{Finite-dimensional autonomous oscillator}\label{sS:num-osc}
Here we consider~\eqref{sys-intro}, with
\begin{equation}\label{Aex-osc}
\clA_\sigma=\begin{bmatrix}0& 1\\-1&\sigma\end{bmatrix},
\end{equation}
and~$(B,C)$ as in~\eqref{BCQ-ode-per}. In the autonomous case the Riccati equations~\eqref{Ricc-sig-feed} are in fact algebraic, $\dot\Pi_\sigma=0$, which we have solved with~$Q=C$.  The uncertain parameter~$\sigma\in\bbR$ corresponds to a damping (with~$\sigma<0$) or energizing (with~$\sigma>0$) parameter. The free dynamics is exponentially stable in the case~$\sigma<0$ and exponentially unstable in the  case~$\sigma>0$. In case~$\sigma=0$ the norm of the state is preserved.

We assume that~$\sigma\in[-1,1]$ and consider the training set
\begin{align}
\varSigma&=\varSigma_{N}\coloneqq\left\{\!\left.i_1\tfrac1{N_1}\,\right| -{N_1}\le i_1\le{N_1}\right\}\subset[-1,1],\qquad N\coloneqq\#\varSigma_N=2N_1+1.\label{Sigma.exp.osc}
\end{align}
with integer~$N_1\ge1$. We shall consider the cases~$N\in\{11,21\}$.

Let the true parameter~$\sigma$ and our initial guess~$\sigma^\tte(0)$ be as
\begin{equation}\label{sigsige-num-osc}
\sigma=0.95\quad\mbox{and}\quad\sigma^\tte(0)=0.
\end{equation}
Note that~$\sigma\notin\varSigma_{21}\supset\varSigma_{11}$, with~$\varSigma_{N}$ as in~\eqref{Sigma.exp.osc}.

 Following the strategy, in each time interval~$\clI_j^\tau=((j-1)\tau,j\tau)$, as subset of training parameters we shall take~$\varSigma^*=\varSigma_\rml\bigcup\varSigma_\rmg$, with local parameters~$\varSigma_\rml\subseteq\varSigma$ in a ball of radius~$\sqrt\gamma$, centered at our latest update and with~$N_\rmg=\#\varSigma_\rmg\subset\varSigma$ randomly chosen equally likely parameters generated using the Matlab function \texttt{rand}; see Algorithm~\ref{Alg:Onl-concat}. The parameter has been updated at time instants~$j\tau$ with~$\tau=0.5$, $j\in\bbN$.

Recall that the free dynamics is unstable for the true parameter~$\sigma=0.95$.
\begin{figure}[ht]
\centering
\subfigure
{\includegraphics[width=0.45\textwidth,height=0.375\textwidth]{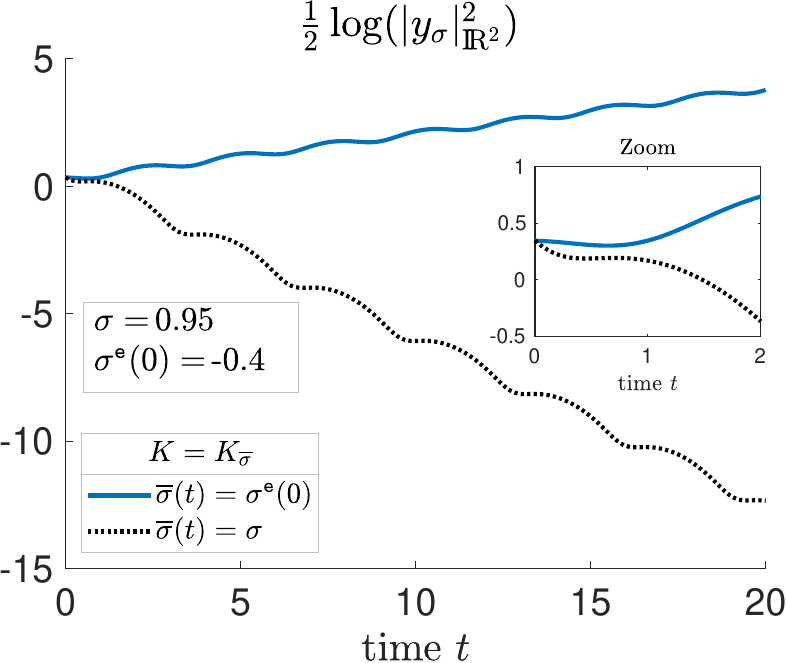}}
\qquad
\caption{$\clA_\sigma$ as in~\eqref{Aex-osc}. No parameter update, $(N_\rmg,\gamma)=(0,0)$.}
\label{fig.osc_noup}
\end{figure}
 In Fig.~\ref{fig.osc_noup} we see that our initial guess is not good enough in order to obtain a stabilizing input (here we have simply run our algorithm with~$(N_\rmg,\gamma)=(0,0)$, to avoid the parameter update). Hence, we need to update our guess/parameter.
\begin{figure}[ht]
\centering
\subfigure
{\includegraphics[width=0.45\textwidth,height=0.375\textwidth]{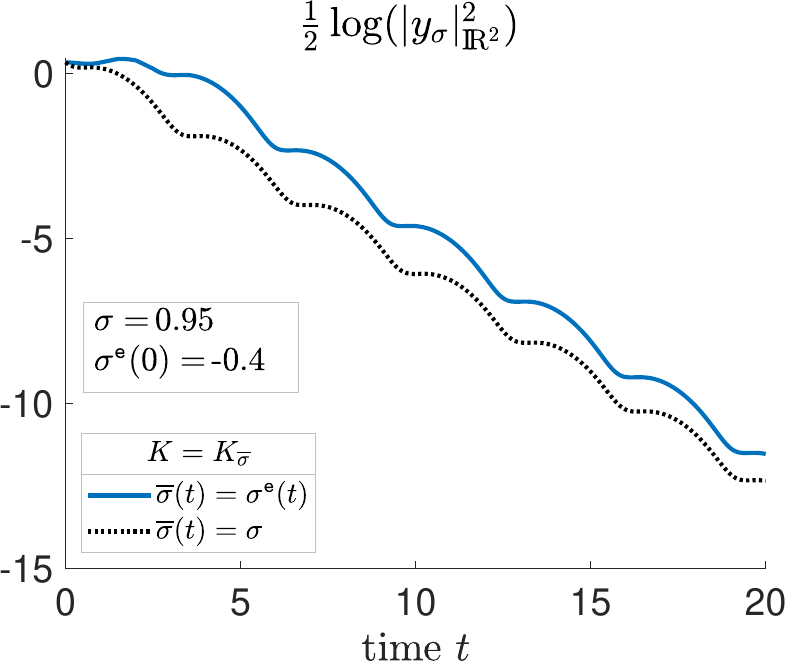}}
\qquad
{\includegraphics[width=0.45\textwidth,height=0.375\textwidth]{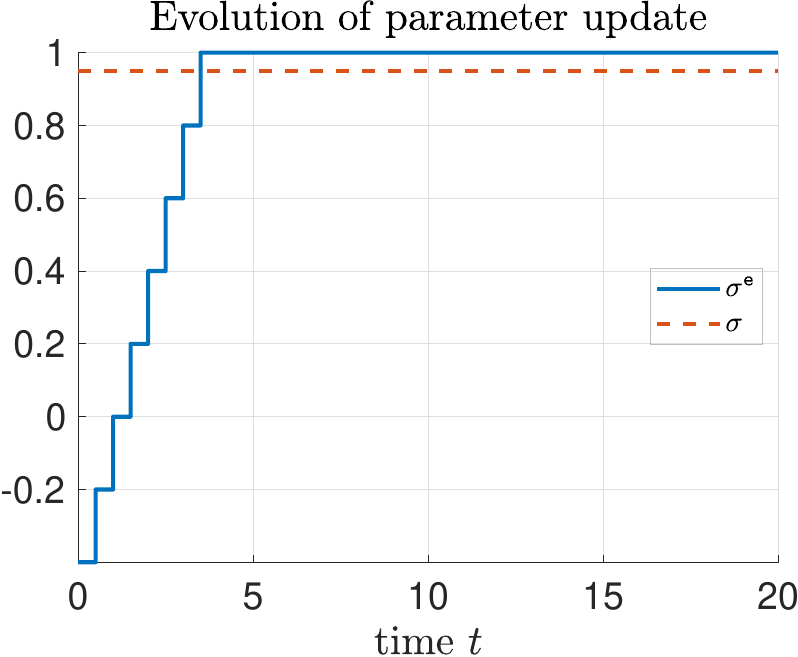}}
\caption{$\clA_\sigma$ as in~\eqref{Aex-osc}.  $\varSigma=\varSigma_{11}$ as in~\eqref{Sigma.exp.osc}, $(N_\rmg,\gamma)=(0,0.1)$.}
\label{fig.osc_train11}
\end{figure}
For that we followed the proposed strategy and present the results in Fig.~\ref{fig.osc_train11}, where we see that the constructed input is stabilizing, where we have taken only local parameters,~$\varSigma^*=\varSigma_\rml$, with~$\gamma=0.1$. Furthermore,  the update reaches and remains the closest to the true parameter after some time.
\begin{figure}[ht]
\centering
\subfigure
{\includegraphics[width=0.45\textwidth,height=0.375\textwidth]{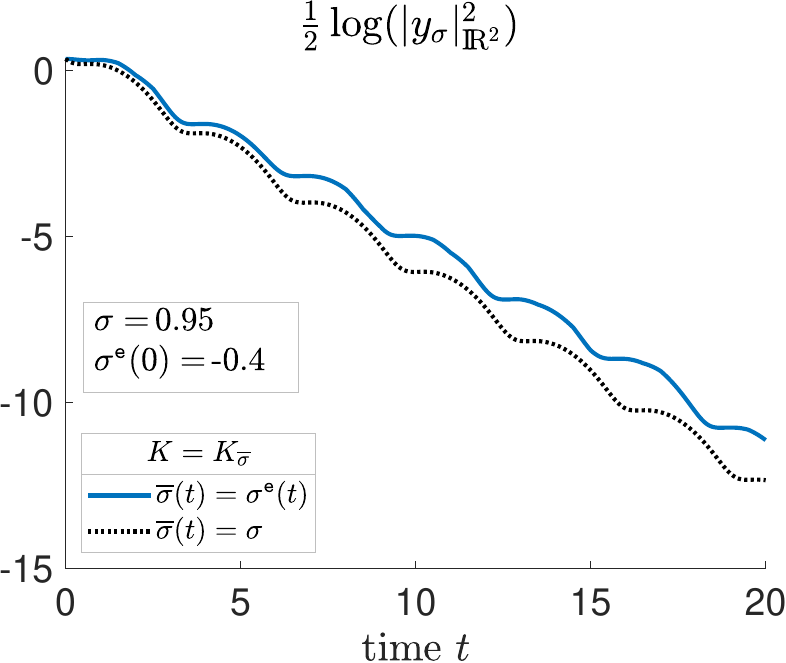}}
\qquad
{\includegraphics[width=0.45\textwidth,height=0.375\textwidth]{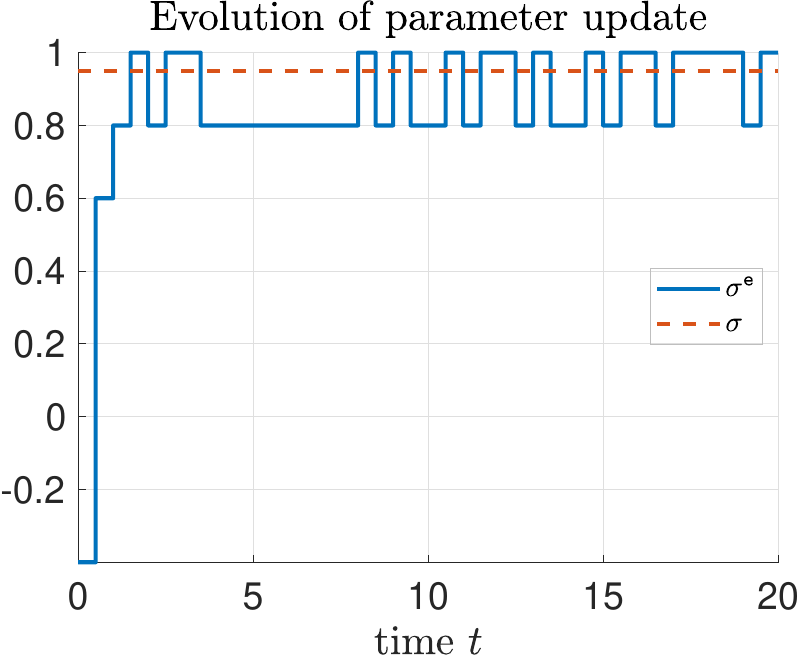}}
\caption{$\clA_\sigma$ as in~\eqref{Aex-osc}.  $\varSigma=\varSigma_{11}$ as in~\eqref{Sigma.exp.osc}, $(N_\rmg,\gamma)=(10,0.1)$.}
\label{fig.osc_train11rand}
\end{figure}
Next, we add $N_\rmg=\#\varSigma_{\rmg}=10$ random global parameters to the training set~$\varSigma^*=\varSigma_\rml\bigcup\varSigma_\rmg$ (re-generated at each time~$t=j\tau$, $j\in\bbN$) and present the results in Fig.~\ref{fig.osc_train11rand}. Again, we obtain a stabilizing input.  Note that, in this case we may have random parameters closer to~$\sigma$, and thus get a closer  update. This can be seen in Fig.~\ref{fig.osc_train11rand} where we reach faster parameters which are close to the true~$\sigma$. Again, the estimate remains close to~$\sigma$ (switching between the closest parameters, in~$\varSigma$, below and above~$\sigma$).

Finally,  in Fig.~\ref{fig.osc_train21} we see the results for the refined training set~$\Sigma_{21}$, where we took only local parameters (as in Fig.~\ref{fig.osc_train11}). Again we obtain a stabilizing input. Furthermore, since we have more parameters in the training subsets~$\varSigma^*=\varSigma_\rml$, we see that  the update~$\sigma^\tte(t)$ moves faster to a neighborhood of~$\sigma$, compared to  Fig.~\ref{fig.osc_train11}.
\begin{figure}[ht]
\centering
\subfigure
{\includegraphics[width=0.45\textwidth,height=0.375\textwidth]{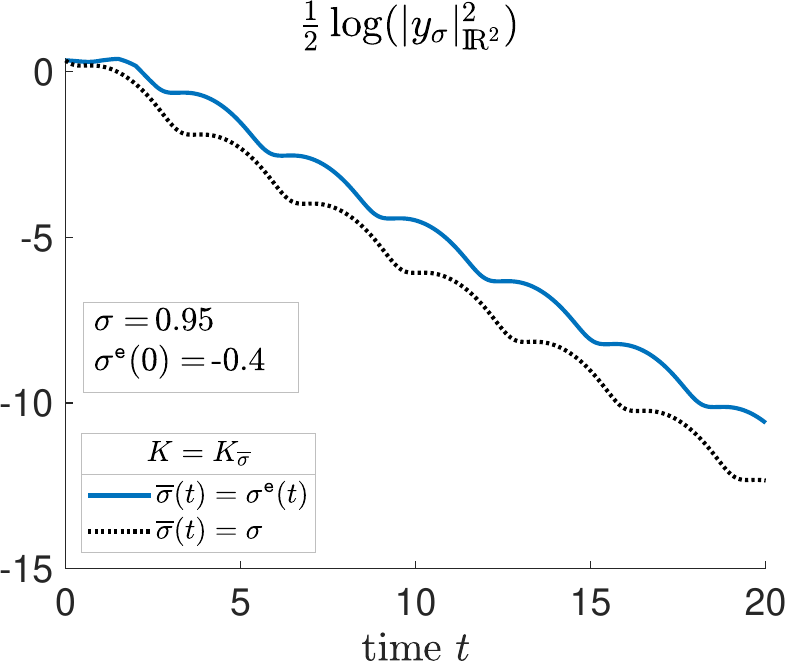}}
\qquad
{\includegraphics[width=0.45\textwidth,height=0.375\textwidth]{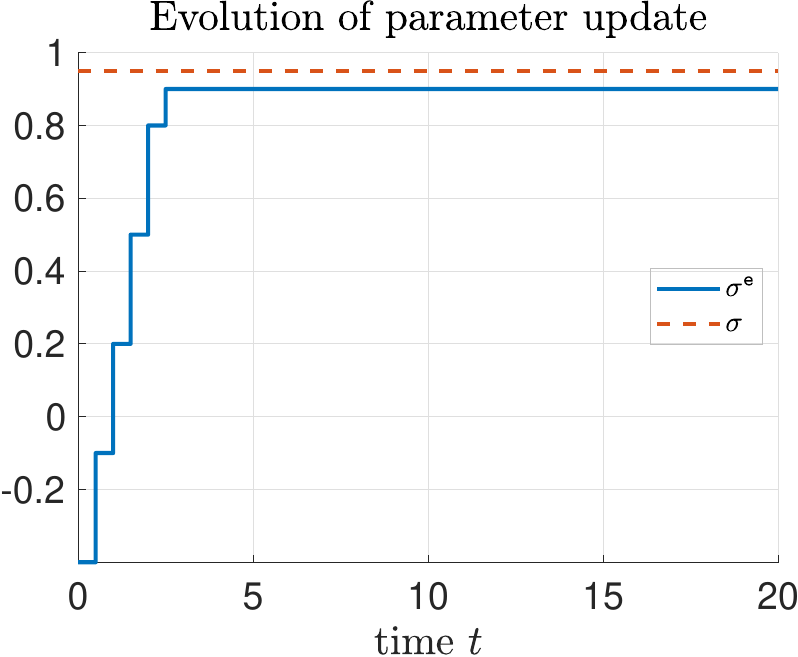}}
\caption{$\clA_\sigma$ as in~\eqref{Aex-osc}.  $\varSigma=\varSigma_{21}$ as in~\eqref{Sigma.exp.osc}, $(N_\rmg,\gamma)=(0,0.1)$.}
\label{fig.osc_train21}
\end{figure}

\medskip
The algebraic Riccati equations were solved using the software in~\cite{BennerRicSolver}; see~\cite{Benner06}.
We solved the state equations with a Crank--Nicolson--Adams--Bashforth  scheme  with temporal step~$t^{\rm step}=10^{-3}$.

\subsection{Finite-dimensional time-periodic dynamics}\label{sS:num-per}
We consider a time-$\rho$ periodic dynamics defined by the matrix
\begin{equation}\label{Aex-tpd}
\clA_{\sigma}(t)=\clA_{(\rho,\phi)}(t)=\Psi(\rho^{-1}t +\phi)\begin{bmatrix}0& 1\\1&0\end{bmatrix},\quad\mbox{with}\quad \Psi(s)\coloneqq1+6\sin(2\pi s),
\end{equation}
and with an uncertain parameter~$\sigma\coloneqq(\rho,\phi)\in\bbR_+\times[0,1)\subset\bbR^2$,
consisting of the time-period~$\rho$ and the time-phase~$\phi$. Note that, given~$\sigma=(\rho,\phi)$, the matrix~$\clA_\sigma$ satisfies
\begin{equation}\notag
\clA_{(\rho,\phi)}(t)=\clA_{(\rho,0)}(t+\rho\phi).
\end{equation}
In this case we can see that if~$\Pi_{(\rho,\phi)}(t)$ stands for the periodic Riccati solutions corresponding to a given uncertain pair~$\sigma=(\rho,\phi)$, then
\begin{equation}\label{Ricc_per_unc}
\Pi_{(\rho,\phi)}(t)=\Pi_{(\rho,0)}(t+\rho\phi).
\end{equation}
Indeed,  since~$\Pi_{(\rho,0)}$ solves~\eqref{Ricc-sig-feed} for the case~$\phi=0$, for~$\Pi^0(t)\coloneqq\Pi_{(\rho,0)}(t+\rho\phi)$, we find
 \begin{align}%
 \dot{\Pi}^0&=\dot\Pi_{(\rho,0)}(t+\rho\phi)\notag\\
 &=-\clA_{(\rho,0)}^*(t+\rho\phi){\Pi}^0(t)-{\Pi}^0(t) \clA_{(\rho,0)}^*(t+\rho\phi)+{\Pi}^0(t)BB^*{\Pi}^0(t)-Q^* Q\notag\\
 &=-\clA_{(\rho,\phi)}^*(t){\Pi}^0(t)-{\Pi}^0(t) \clA_{(\rho,\phi)}^*(t)+{\Pi}^0(t)BB^*{\Pi}^0(t)-Q^* Q.\notag
 \end{align}
Hence, ~$\Pi^0$ solves~\eqref{Ricc-sig-feed} for the given parameter~$\sigma=(\rho,\phi)$. Since~$\Pi_{(\rho,\phi)}(t)$ solves the same equation, by the uniqueness of the solution (due to the assumed stabilizability and detectability in~\eqref{assum-StabDete}) it follows that~\eqref{Ricc_per_unc} holds true.

In other words, the  optimal dynamics is, for a fixed~$(\rho,\phi)$, given by
\begin{equation}\label{sys_per_sigrsige}
\dot y=\clA_{(\rho,\phi)} y-BB^\top\Pi_{(\rho,0)}(t+\rho\phi)y.
\end{equation}
Consequently we need to construct and store the feedback for parameters as~$\sigma=(\rho,0)$. Let us assume that we know that the time period satisfies~$\rho\in[0.5,1.5]$. Then we take the training set as
\begin{subequations}\label{Sigma.exp.per}
\begin{align}
\varSigma&=\varSigma_{N}\subset[0.5,1.5]\times[0,1),\qquad N\coloneqq\#\varSigma_N=(2N_1+1)\times N_2;
\intertext{with fixed positive integers~$N_1$ and~$N_2$, and}
\varSigma_{N}&\coloneqq\left\{\!\left.\left(1+i_1\tfrac1{2N_1}, i_2\tfrac1{\overline N_2}\right)\,\right| -{N_1}\le i_1\le{N_1}, 0\le i_2\le {N_2-1}\right\}.
\end{align}
\end{subequations}
We present the results of experiments for the cases~$N\coloneqq(N_1,N_2)\in\{(10,30), (20,60)\}$.

The control, output, and Riccati observation operators were taken as
\begin{equation}\label{BCQ-ode-per}
B=\begin{bmatrix}0 \\ 1\end{bmatrix},\quad C=\begin{bmatrix}1 & 0\end{bmatrix},\quad\mbox{and}\quad Q=\begin{bmatrix}1&0\\0 & 1\end{bmatrix}.
\end{equation}

The true parameter~$\sigma$ and our initial guess~$\sigma^\tte(0)$ are taken as,
\begin{equation}\label{sigsige-num}
\sigma=(\rho,\phi)=(1.47,0.51)\quad\mbox{and}\quad\sigma^\tte(0)=(\rho^\tte(0),\phi^\tte(0))=(1,0).
\end{equation}

\subsubsection{Instability of free dynamics and need of a parameter update} Fig.~\ref{fig.per_free}
\begin{figure}[ht]
\centering
\subfigure[Free dynamics.\label{fig.per_free}]
{\includegraphics[width=0.45\textwidth,height=0.375\textwidth]{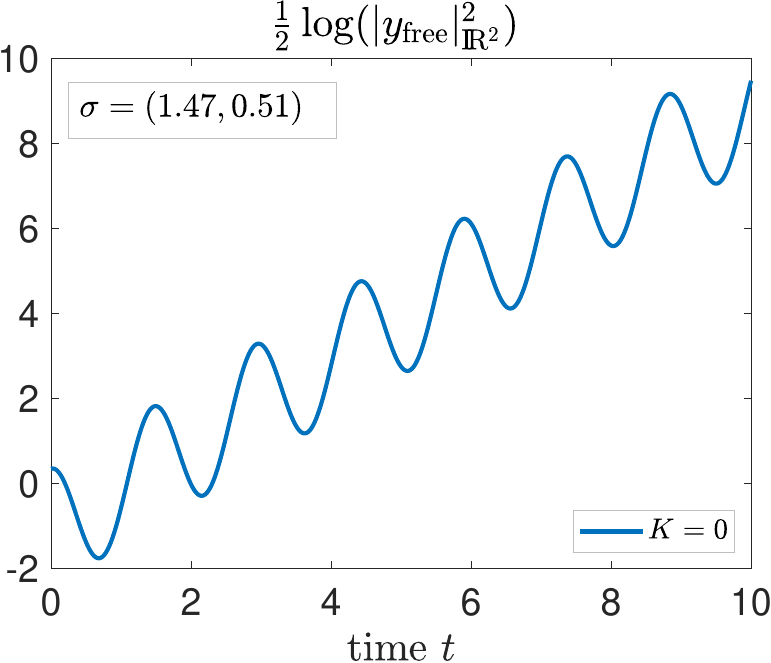}}
\qquad
\subfigure[No parameter update, $(N_\rmg,\gamma)=(0,0)$.\label{fig.per_noup}]
{\includegraphics[width=0.45\textwidth,height=0.375\textwidth]{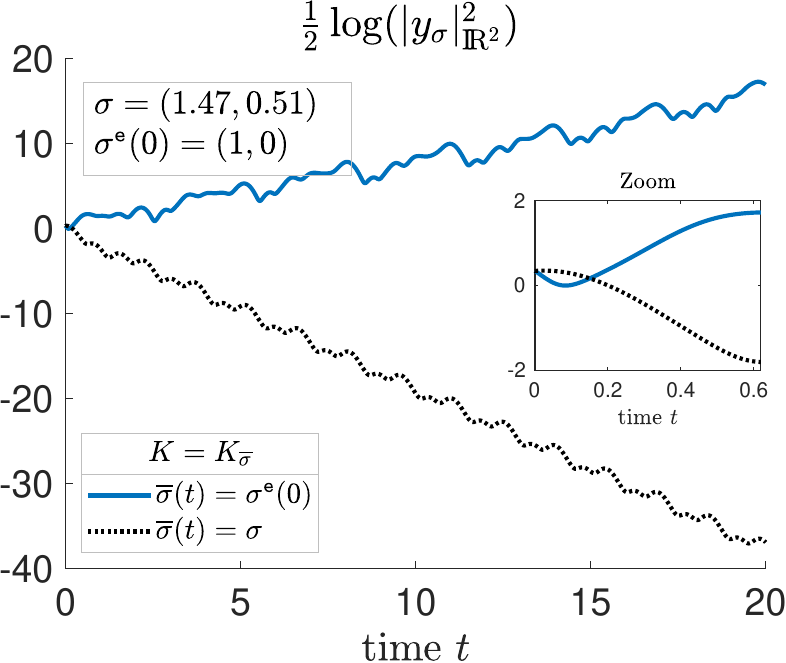}}
\caption{$\clA_\sigma$ as in~\eqref{Aex-tpd}. Instability with no control and with no feedback control update.}
\label{fig.per_free_and_noup}
\end{figure}
shows that the free dynamics is exponentially unstable, thus a control is needed to drive the system asymptotically to zero. Fig.~\ref{fig.per_noup} shows that the feedback input operator~$K_{\sigma^\tte(0)}$, corresponding to our initial guess~$\sigma^\tte(0)\in\varSigma$, does not provide us with a stabilizing input, thus an update of our initial guess is needed.

\subsubsection{Performance following the parameter update}
We present the results, for the input obtained with the proposed parameter updating strategy. We have chosen the update finite time-horizon as~$\tau=0.1$.
In each time interval~$\clI_j^\tau=((j-1)\tau,j\tau)$, as subset of training parameters we have taken~$\varSigma^*=\varSigma_\rml\bigcup\varSigma_\rmg$, with the local parameters~$\varSigma_\rml\subseteq\varSigma$ in a ball of radius~$\gamma=0.02$ centered at our latest update and with  no random global parameters,~$N_\rmg=\#\varSigma_\rmg=0$; see Algorithm~\ref{Alg:Onl-concat}.
\begin{figure}[ht]
\centering
\subfigure
{\includegraphics[width=0.45\textwidth,height=0.375\textwidth]{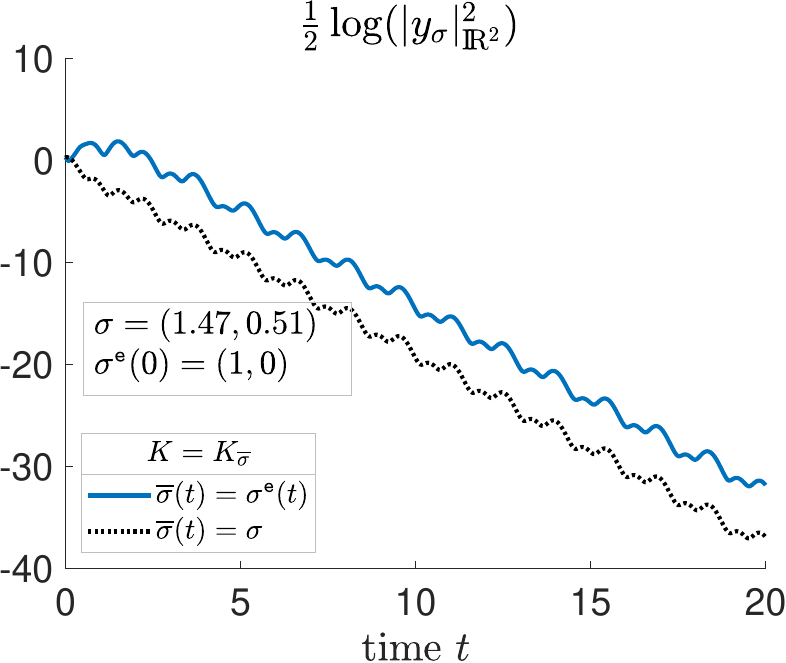}}
\qquad
\subfigure
{\includegraphics[width=0.45\textwidth,height=0.375\textwidth]{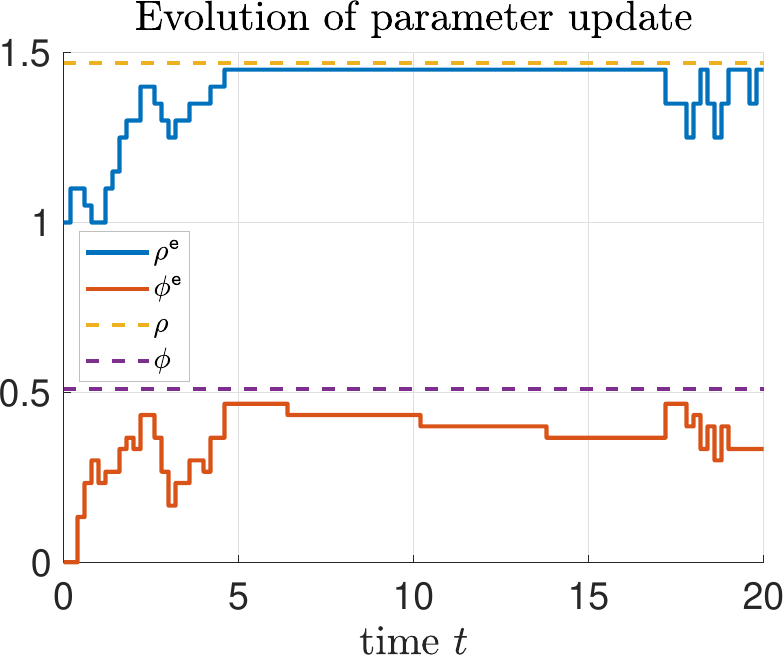}}
\caption{$\clA_\sigma$ as in~\eqref{Aex-tpd}. $\sigma\notin\varSigma=\varSigma_{(10,30)}$ as in~\eqref{Sigma.exp.per}, $(N_\rmg,\gamma)=(0,0.02)$.}
\label{fig.per_2130g0}
\end{figure}
In Fig.~\ref{fig.per_2130g0}
we see that the proposed strategy leads to a stabilizing control input. Further, the stabilization rate is (with the naked eye) the same as the one we would obtain in case we knew~$\sigma$ and used the corresponding Riccati feedback~$K_\sigma$.
\begin{figure}[ht]
\centering
\subfigure
{\includegraphics[width=0.45\textwidth,height=0.375\textwidth]{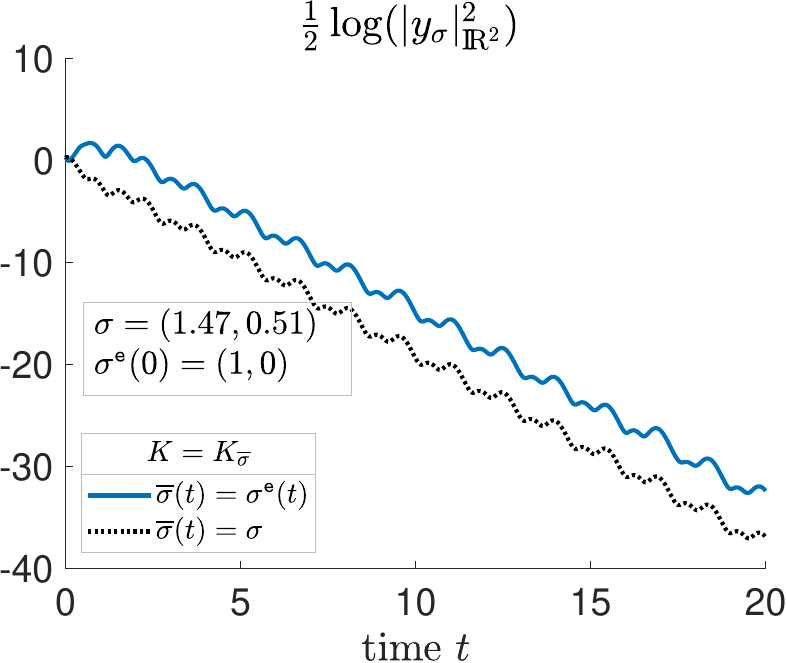}}
\qquad
\subfigure
{\includegraphics[width=0.45\textwidth,height=0.375\textwidth]{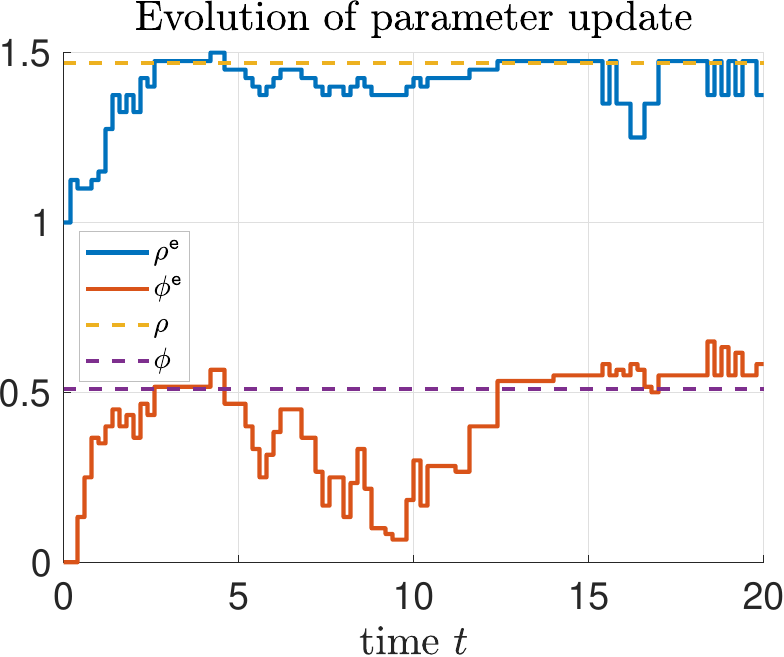}}
\caption{$\clA_\sigma$ as in~\eqref{Aex-tpd}. $\sigma\notin\varSigma=\varSigma_{(20,60)}$ as in~\eqref{Sigma.exp.per}, $(N_\rmg,\gamma)=(0,0.02)$.}
\label{fig.per_4160g0}
\end{figure}
In Fig.~\ref{fig.per_4160g0} we have taken a refined training set~$\varSigma$. The corresponding constructed  update, leads us to stabilizing properties for the control input which are analogous to those observed in Fig.~\ref{fig.per_2130g0}.
These two figures  show that we obtain a stabilizing input, though the obtained update~$\sigma^\tte$ is not necessarily the closest to the true unknown parameter~$\sigma$. We can also confirm the difficulties/inability on the identification of the parameter~$\sigma=(\rho,\phi)$ we have discussed in Example~\ref{Ex:per1}.
 Motivated by that discussion, we would like to know the error on the identification of the matrix~$\clA_\sigma$ characterizing the dynamics, instead.

  We plot the results in Fig.~\ref{fig.per-ode-difAsig}, utilizing the function~$\Psi$ defined in~\eqref{Aex-tpd}.
\begin{figure}[ht]
\centering
\subfigure
{\includegraphics[width=0.45\textwidth,height=0.375\textwidth]{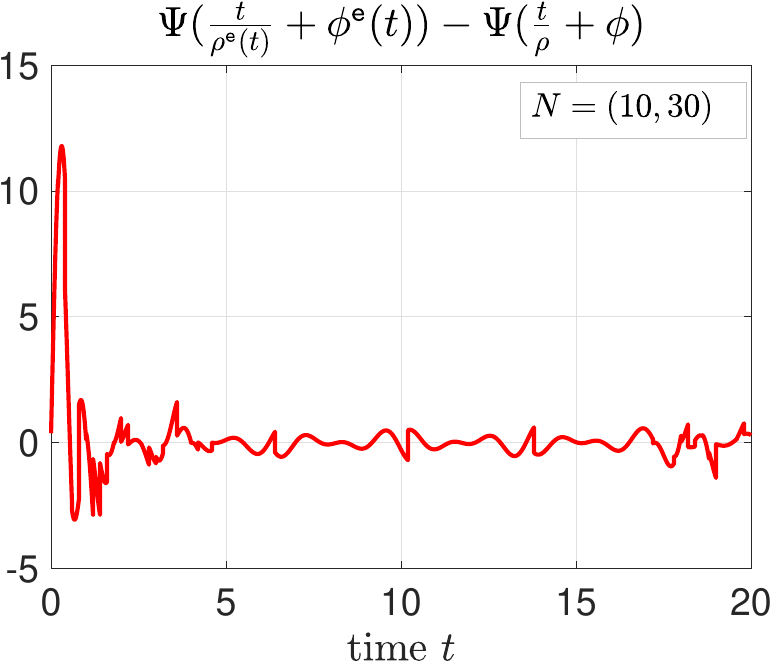}}
\qquad
\subfigure
{\includegraphics[width=0.45\textwidth,height=0.375\textwidth]{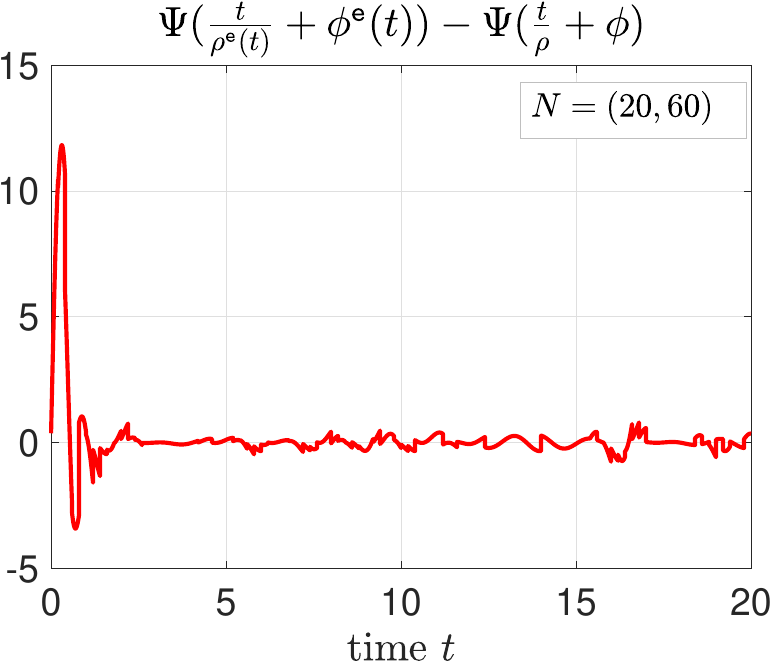}}
\caption{Dynamics identification. $\sigma\notin\varSigma_N$ as in~\eqref{Sigma.exp.per}, $(N_\rmg,\gamma)=(0,0.02)$.}
\label{fig.per-ode-difAsig}
\end{figure}
 We see that this error reaches and remains in a neighborhood of zero after some time, even if the parameter mismatch, in Figs.~\ref{fig.per_2130g0} and~\ref{fig.per_4160g0}, is not small.  We confirm that the strategy will attempt at estimating~$\clA_\sigma$ rather than~$\sigma$.

\begin{figure}[ht]
\centering
\subfigure
{\includegraphics[width=0.45\textwidth,height=0.375\textwidth]{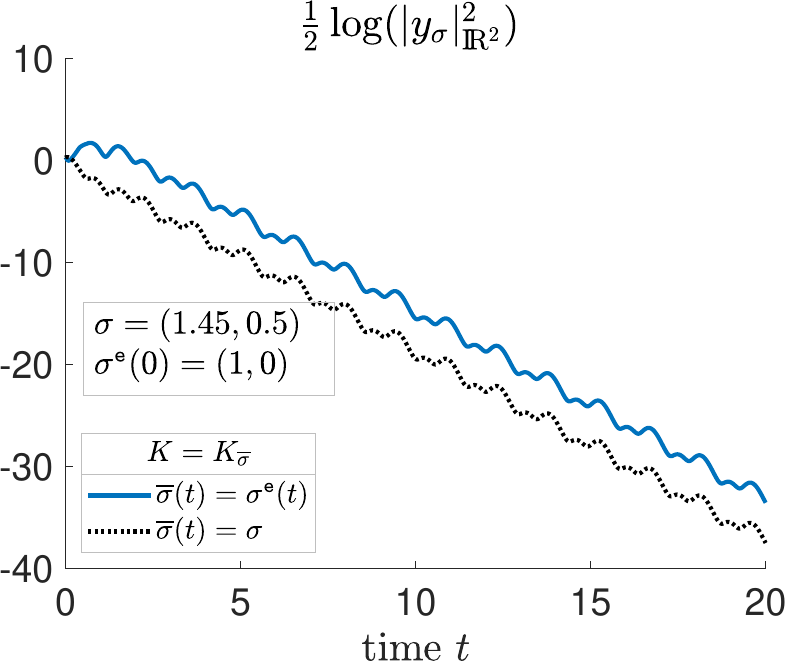}}
\qquad
\subfigure
{\includegraphics[width=0.45\textwidth,height=0.375\textwidth]{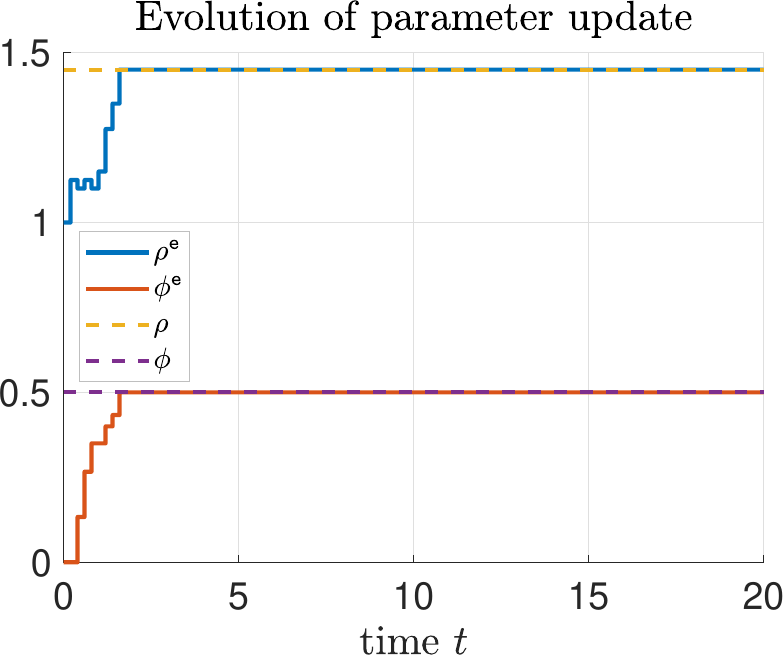}}
\caption{$\clA_\sigma$ as in~\eqref{Aex-tpd}. $\sigma\in\varSigma=\varSigma_{(20,60)}$ as in~\eqref{Sigma.exp.per}, $(N_\rmg,\gamma)=(0,0.02)$.}
\label{fig.per_4160g0_IN}
\end{figure}
Finally, in Fig.~\ref{fig.per_4160g0_IN} we consider the case where the true parameter belongs to the training set~$\varSigma=\varSigma_{(20,60)}$, namely, we take~$\sigma=(1.475,0.5)$. Note that, this is a perturbation of the parameter~$(1.47,0.51)$ that we have taken before in Fig.~\ref{fig.per_4160g0}. But, now the minimum in~\eqref{sig-update-intro-fkC} equals~$0$ (exactly), as soon as~$\sigma\in\varSigma^*$. In Fig.~\ref{fig.per_4160g0}, we see that the parameter update reaches the exact value of~$\sigma$ after a few number of updates,  and also that future updates remain at~$\sigma$.
Further, graphically,  the stabilizing performances in Figs.~\ref{fig.per_2130g0}, \ref{fig.per_4160g0}, and~\ref{fig.per_4160g0_IN} are similar. This is expected for dynamics which are close to each other.

\subsubsection{Robustness against state estimation errors}
We return again to the setting in Fig.~\ref{fig.per_4160g0} where~$\sigma=(1.47,0.51)$ is not in the training set~$\varSigma=\varSigma_{(20,60)}$. We want to check the robustness of the strategy against state measurement errors at the concatenating times~$t=j\tau$, $j\in\bbN$. Note that the auxiliary systems in~\eqref{sys-Feed-intro-aux} are solved assuming the exact knowledge of the state at these time instants; see step~\ref{Alg:Onl.run-auxil}~\ref{Alg:Onl.run-auxil1} in Algorithm~\ref{Alg:Onl-1int}. In applications we may need to use an estimate of this state, so we will consider the case where we know these states up to a measurement error, namely, we let the true system~\eqref{sys-Feed-intro}  (i.e., with the true parameter~$\sigma$) run for time~$t\in I_j^{\tau}=((j-1)\tau,j\tau)$ with its initial state~$y^{\rm old}\coloneqq y_{\sigma}((j-1)\tau)$, but we solve the auxiliary systems, in Algorithm~\ref{Alg:Onl-1int}, with the ``measured'' initial state, at~$t=t_0=(j-1)\tau$, as
\begin{equation}
y_{\sigma_i^*}((j-1)\tau)=y^{\rm old}+\eta_{\rm mag}(v_{\tt rand}(j)-0.5);\qquad\quad y^{\rm old}= y_{\sigma}((j-1)\tau).
\end{equation}
Here~$v_{\tt rand}(j)\in\bbR^{2\times1}$ is a column vector with random entries uniformly distributed in~$[0,1]$; in simulations, this vector  was generated by the Matlab function~$\tt rand$. The scalar~$\eta_{\rm mag}$ sets the magnitude of the measurement error. In Fig.~\ref{fig.noise_ode},
we find the results for several~$\eta_{\rm mag}= \xi_1\cdot10^{\xi_2}$ and see that our strategy is robust against such state error measurements,  in the sense that the constructed input will drive the state to (and keep it in) a ball centered at zero with radius~$R_{\eta_{\rm mag}}$ depending on~$\eta_{\rm mag}$ and, furthermore,~$R_{\eta_{\rm mag}}$ likely converges to~$0$ when~$\eta_{\rm mag}$ does.
\begin{figure}[ht]
\centering
\subfigure
{\includegraphics[width=0.45\textwidth,height=0.375\textwidth]{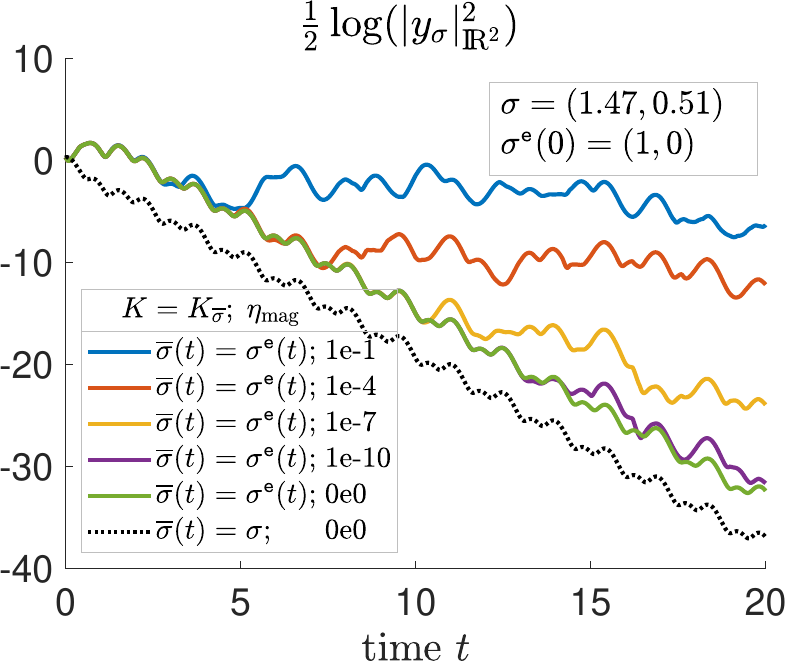}}
\caption{$\clA_\sigma$ as in~\eqref{Aex-tpd}. $\sigma\notin\varSigma=\varSigma_{(20,60)}$ as in~\eqref{Sigma.exp.per}, $(N_\rmg,\gamma)=(0,0.02)$.}
\label{fig.noise_ode}
\end{figure}

\smallskip
The time-periodic Riccati equations were solved using a Crank--Nicolson scheme as proposed in~\cite[sect.~3.5]{Rod23-eect}, with temporal-step~$t^{\rm step,ric}=10^{-2}$.
As before, the dynamical systems were solved by a Crank--Nicolson--Adams--Bashforth scheme  with temporal step~$t^{\rm step}=10^{-3}$.

\subsection{Parabolic time-periodic dynamics}\label{sS:num-PDE}
Here we consider the case where the state is an element of an infinite-dimensional space. Namely, we consider the dynamics given by a time-periodic diffusion-reaction-convection parabolic equation. We assume that the time-period and the time-phase are known, but we are uncertain about a parameter~$\sigma\in[0,2\pi)$ in the convection term.  Namely, we consider system~\eqref{sys-intro} with
\begin{subequations}\label{Aex-parab}
\begin{align}
 \clA_\sigma y&\coloneqq-(-\nu\Delta+\Id) y-a y -b_\sigma\cdot\nabla y\\
 Bu&\coloneqq\sum_{j=1}^m u_j\indf_{\omega_j},\qquad Cy\coloneqq \Xi P_{\clE_p^\rmf}y.
   \end{align}
\end{subequations}
The state~$y=y(t,x)$ is defined for time~$t>0$ and~$x=(x_1,x_2)\in\Omega\coloneqq(0,1)\times(0,1)$. The initial state~$y_0$ and the diffusion coefficient~$\nu$, were taken as
\begin{align} \label{inidataPDE}
y(0,x)&=y_0(x)\coloneqq1-3x_1\sin(x_1),\qquad \nu\coloneqq0.1.
 \end{align}
  We take the reaction~$a$ and convection~$b_\sigma$ coefficients as follows,
\begin{subequations}\label{dataSystem}
\begin{align}
a(x)&=-(1+5\sin(2\pi t))(1+x_1),\qquad
b_\sigma=\begin{bmatrix}-\tfrac12\cos(\sigma)\\ \sin(\sigma)\end{bmatrix}.
\end{align}
\end{subequations}
 with an uncertain parameter~$\sigma$. In particular, we are uncertain about the direction~$b_\sigma$ the state is being transported in.
We chose Neumann boundary conditions, that is,~$\Delta$ is understood as the Neumann Laplacian with domain~$\rmD(\Delta)$ satisfying
\begin{align}
\rmD(\Delta)=\{h\in W^{2,2}(\Omega)\mid \tfrac{\p h}{\p\bfn}\rest{\p\Omega}=0\}\subset V\coloneqq W^{1,2}(\Omega)\subset H\coloneqq L^{2}(\Omega),
 \end{align}
 where~$\bfn$ stands for the unit outward normal vector to the boundary~$\p\Omega$ of~$\Omega$.

The control operator maps a control input~$u(t)\in\bbR^m$ into the linear span of indicator functions $\indf_{\omega_j}$ of given subsets~$\omega_j\subseteq\Omega$. The output is given by the coordinates of
 the orthogonal projection of the state onto the linear span~${\clE_p^\rmf}=\linspan\{e_j\mid 1\le j\le p\}$ of ``the'' first~$p$ Neumann eigenfunctions~$e_j=e_j(x)$ of the Laplacian. More precisely, $P_{\clE_p^\rmf}$ stands for the orthogonal projection operator onto~${\clE_p^\rmf}$ and~$\Xi$ denotes the operator giving us the coordinates of such projection, that is,
 \begin{equation}\notag
 z=\Xi w\quad\Longleftrightarrow\quad w=\sum_{j=1}^p z_i e_i.
 \end{equation}

We consider~$H$ as a pivot space,~$H'=H$, and we  can see that~$\clA_\sigma\colon V\to V'$ maps~$V$ into its continuous dual~$V'$.
 We take controls~$u\in L^2_{\rm loc}(\bbR_+,\bbR^m)$. Then the state evolves in the Hilbert space~$H$.
 We fix a  training  parameters set as
\begin{equation}\label{Sigma.exp.perPDE}
\varSigma=\varSigma_{N}\coloneqq\left\{\sigma_i=i\tfrac{2\pi}{\overline N}\mid 0\le i\le {N-1}\right\}\subset[0,2\pi).
\end{equation}
Note that~$\clA_\sigma$ is time periodic with period~$\rho=1$. For each parameter~$\sigma_i\in\varSigma$, following~\cite{Rod23-eect}, we solve the Riccati equation~\eqref{Ricc-sig-feed} with~$Q=\Id$ in a coarse spatial triangulation~$\clT^{[0]}$ of~$\Omega$, with~$n^{[0]}$ points and with temporal step~$t^{\rm step}_0=t^{\rm step,ric}$. The corresponding feedback input operators~$K_{\sigma_i}(t)=-B^*\Pi_{\sigma_i}(t)\in\bbR^{m\times n^{[0]}}$ are stored in a library~$\fkL_\varSigma$, for each time instant~$t=kt^{\rm step}_0$, $0\le k<(t^{\rm step}_0)^{-1}$ in a/the discrete mesh of~$[0,1)$.

The auxiliary systems are solved in a spatially refined mesh~$\clT^{[1]}$  with~$n^{[1]}$ points, obtained after~$1$ regular refinement of~$\clT^{[0]}$, and with time step~$t^{\rm step}_1$.

The true system is run in a further refined spatial mesh~$\clT^{[2]}$  with~$n^{[2]}$ points, obtained after~$2$ consecutive regular refinements of~$\clT^{[0]}$, and with time step~$t^{\rm step}_2$.

We may think of the true system running in continuous time, whereas the fictitious auxiliary systems are solved numerically. For this reason, we took a coarser mesh for the auxiliary systems. Further, we have also stored the feedback operators corresponding to the coarsest spatial mesh; this possibility could be important in case we have constraints on the available storage, but it is also a fact that solving the Riccati equations in a fine grid can be a challenging numerical task.

Thus, we fix a coarse spatial triangulation~$\clT^{[0]}$ of~$\Omega$, and consider three levels of spatio-temporal refinements~$\bfR=(\clT^{[{\tt r}]},t^{\rm step})$, $0\le{\tt r}\le2$, namely,
\begin{equation}\notag
\bfR_{\rm ric}=(\clT^{[0]},{10^{-2}}),\qquad\bfR_{\rm aux}=(\clT^{[1]},{10^{-3}}),\qquad\bfR_{\rm true}=(\clT^{[2]},{10^{-4}}).
\end{equation}
where we solve, respectively, the Riccati equations~\eqref{Ricc-sig-feed}, the auxiliary systems~\eqref{sys-Feed-intro-aux} and the true system~\eqref{sys-Feed-intro}.
 The spatial triangulations are shown in Fig.~\ref{fig.mesh0}, where we also show the support~$\omega_i$ of the~$m=4$ actuators.
 \begin{figure}[ht]
\centering
\subfigure
{\includegraphics[width=0.3\textwidth,height=0.25\textwidth]{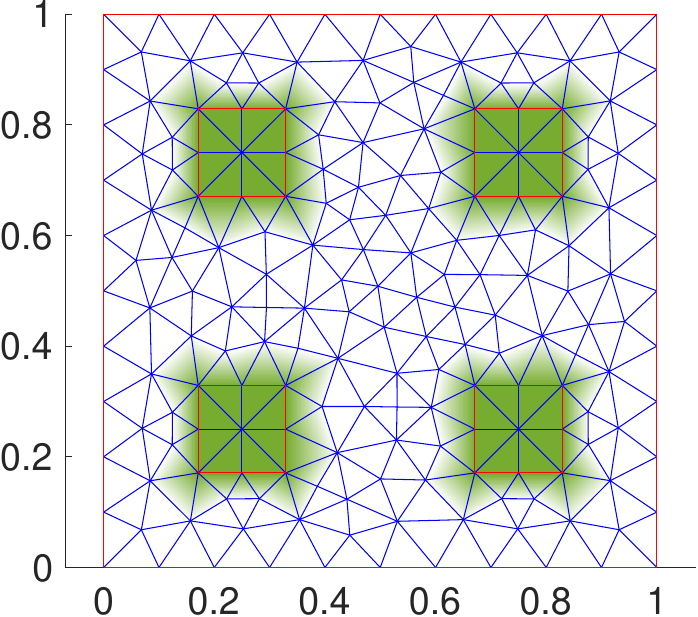}}
\quad
{\includegraphics[width=0.3\textwidth,height=0.25\textwidth]{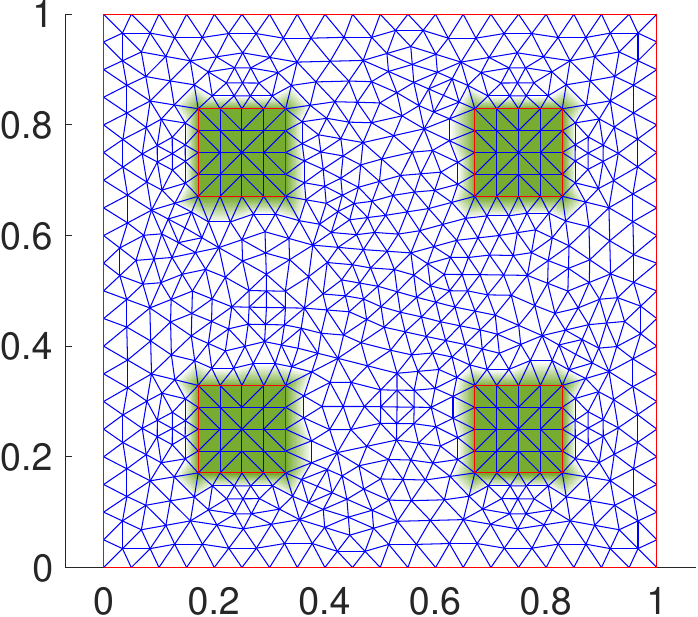}}
\quad
{\includegraphics[width=0.3\textwidth,height=0.25\textwidth]{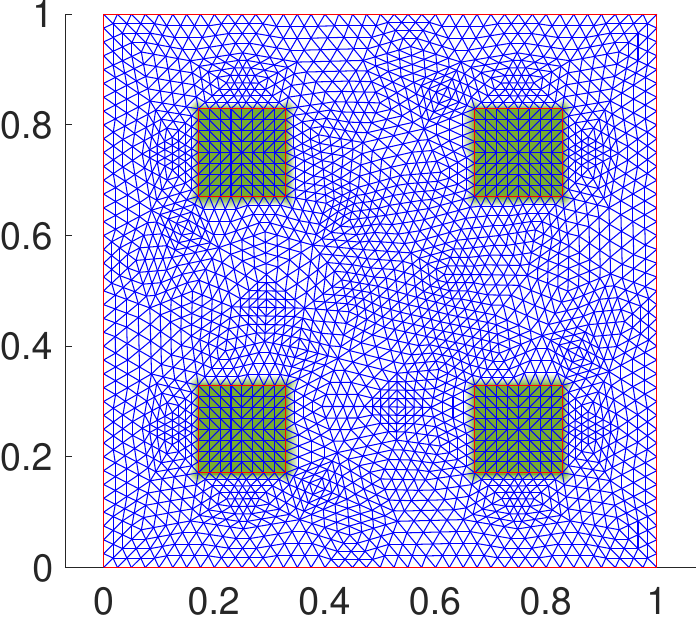}}
\caption{Spatial triangulations and supports of actuators.}
\label{fig.mesh0}
\end{figure}

Let us assume that the true parameter~$\sigma$ and our initial guess~$\sigma^\tte(0)$ are
\begin{equation}\label{sigsige-num-PDE}
\sigma=0.7\quad\mbox{and}\quad\sigma^\tte(0)=0.
\end{equation}

We have chosen the updating time-horizon as~$\tau=0.1$.

\subsubsection{Instability of free dynamics and need of a parameter update}
Fig.~\ref{fig.PDE_free} shows  that the free dynamics is unstable and Fig.~\ref{fig.PDE_noup}  shows that the initial guess~$\sigma^\tte(0)$ is not good enough in order to obtain a stabilizing input, hence in order to stabilize the system, an input is needed as well as an update of our guess. The result shown in Fig.~\ref{fig.PDE_noup} corresponds to a run of our algorithm with~$(N_\rmg,\gamma)=(0,0)$.
\begin{figure}[ht]
\centering
\subfigure[Free dynamics.\label{fig.PDE_free}]
{\includegraphics[width=0.45\textwidth,height=0.375\textwidth]{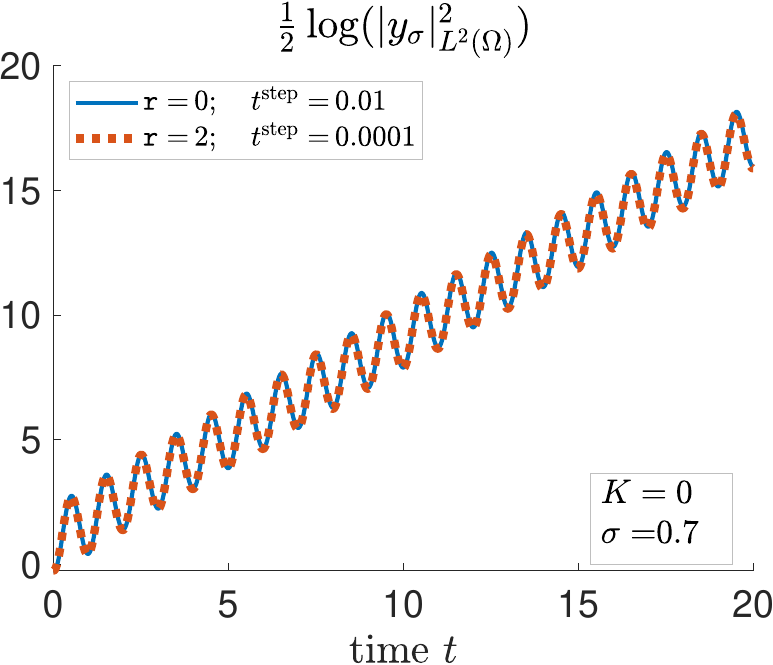}}
\qquad
\subfigure[No parameter update, $(N_\rmg,\gamma)=(0,0)$.\label{fig.PDE_noup}]
{\includegraphics[width=0.45\textwidth,height=0.375\textwidth]{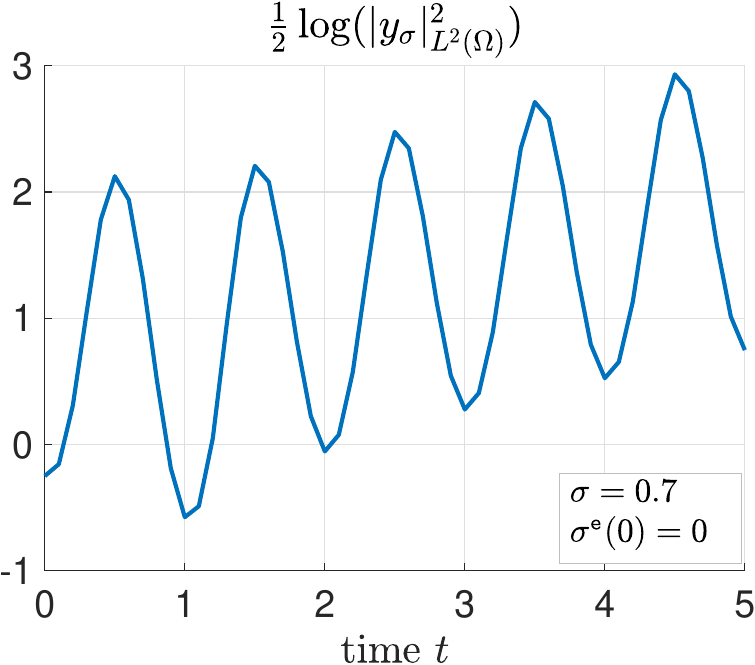}}
\caption{$\clA_\sigma$ as in~\eqref{Aex-parab}. Instability with no control and with no feedback control update.}
\label{fig.PDE_free_and_noup}
\end{figure}
\begin{remark}
The results in Fig.~\ref{fig.PDE_free} are shown for the coarsest and the finest spatial triangulations~$r\in\{0,2\}$, these results show that with the coarsest mesh we are already able to accurately capture the evolution  of the norm of the state for the free dynamics.
\end{remark}

\subsubsection{Using the proposed parameter update strategy}
By following the proposed parameter update strategy, with~$\varSigma=\varSigma_8$ as in~\eqref{Sigma.exp.perPDE} we are provided with a stabilizing feedback input, as shown in Fig.~\ref{fig.PDE_updated}. Furthermore, note that in this test we have taken~$\gamma=1$, which allows the  element~$\tfrac{\pi}{4}=1\cdot\tfrac{2\pi}{8}\in\varSigma_8$ to be in the training subset~$\varSigma^*$ in the first updating time interval~$I_1^\tau=(0,\tau)$. Note that this is the closest element, in~$\varSigma_8$, to the true parameter~$\sigma=0.7$. We see that at the end of this interval the parameter is updated to~$\tfrac{\pi}{4}$, and future updates take the same value. So in this example, besides providing us with a stabilizing input the strategy is also able to identify the closest training parameter to~$\sigma$.
\begin{figure}[ht]
\centering
\subfigure
{\includegraphics[width=0.45\textwidth,height=0.375\textwidth]{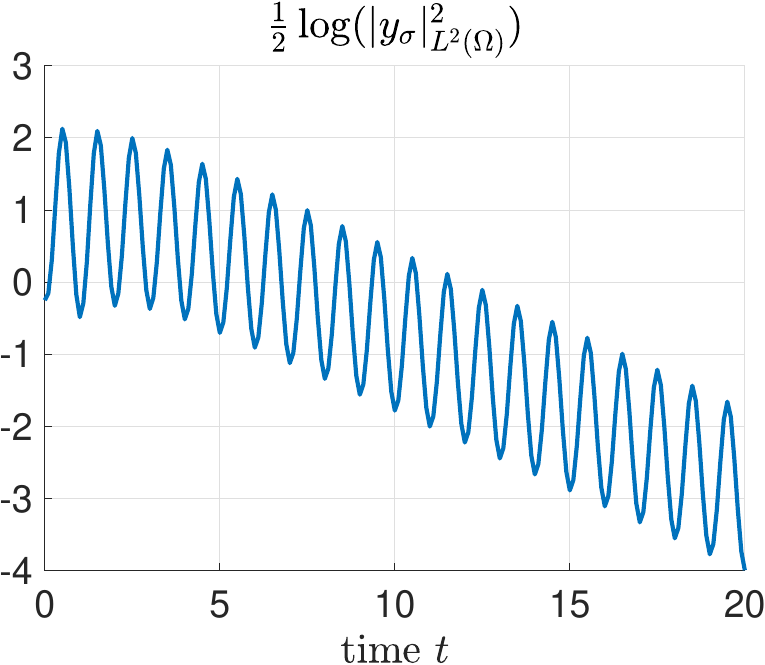}}
\qquad
\subfigure
{\includegraphics[width=0.45\textwidth,height=0.375\textwidth]{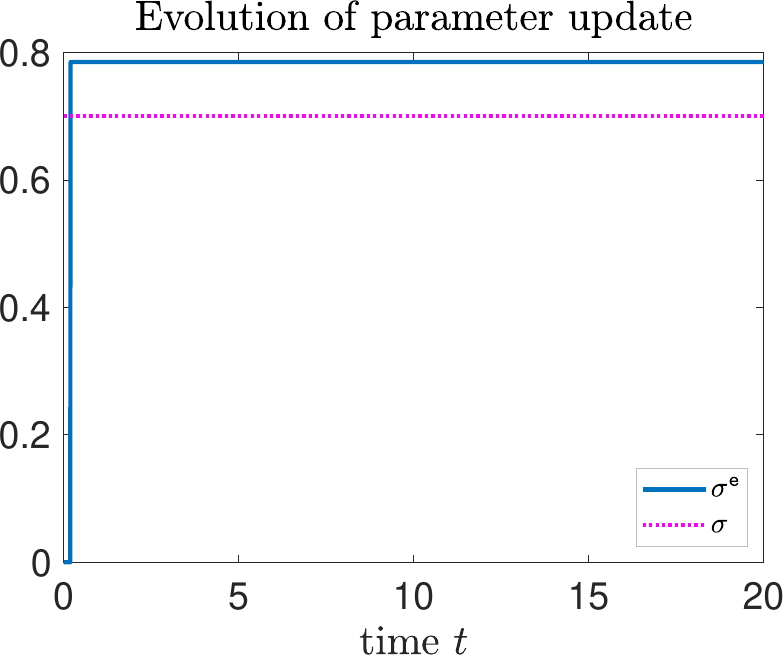}}
\caption{With updated parameter, $(N_\rmg,\gamma)=(0,1)$.}
\label{fig.PDE_updated}
\end{figure}

\subsubsection{On the response to switching true parameters}\label{sec:switching}
We show that the strategy is able to respond to changes in the true parameter. Thus, we can apply it to the case of a  piecewise constant true parameter~$\sigma=\sigma(t)$. However, to achieve stabilization we may need the dwell times  to be large enough. Recall that switching between stable systems can result in an unstable system, see~\cite[Prop.~1]{AkarPaulSafoMitra06} and~\cite[Prob.~A]{LiberzonMorse99}.

The response to parameter changes is confirmed in Fig.~\ref{fig.PDE_sigma_t}, where we can see that the constructed input is stabilizing. Furthermore, in this example the strategy is again able to identify the closest training parameters (below and above) to~$\sigma(t)$.
\begin{figure}[ht]
\centering
\subfigure
{\includegraphics[width=0.45\textwidth,height=0.375\textwidth]{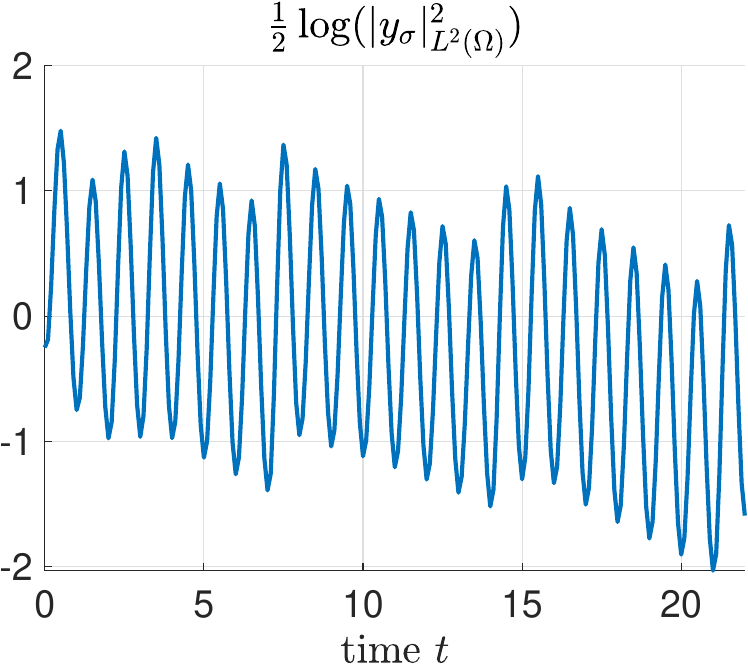}}
\qquad
\subfigure
{\includegraphics[width=0.45\textwidth,height=0.375\textwidth]{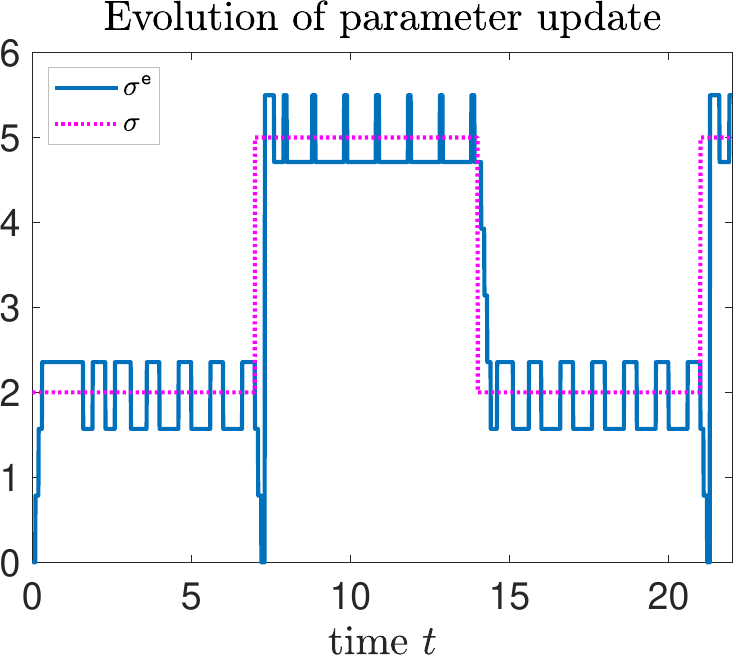}}
\caption{Piecewise constant, slowly switching, parameters $\sigma(t)\notin\varSigma$.}
\label{fig.PDE_sigma_t}
\end{figure}
Note, however, that at the switching times~$t\in\{7,14,21\}$ we observe a transient period where the norm of the state increases in a relevant manner. This suggests that for the family of stable systems~$\{\clA_\varsigma^K=\clA_\varsigma+BK_\varsigma\mid \varsigma\in\varSigma\}$, we may indeed loose stability under fast switching parameters.  This is confirmed in Fig.~\ref{fig.PDE_sigma_t_real} where we have taken
shorter dwell times.
\begin{figure}[ht]
\centering
\subfigure
{\includegraphics[width=0.45\textwidth,height=0.375\textwidth]{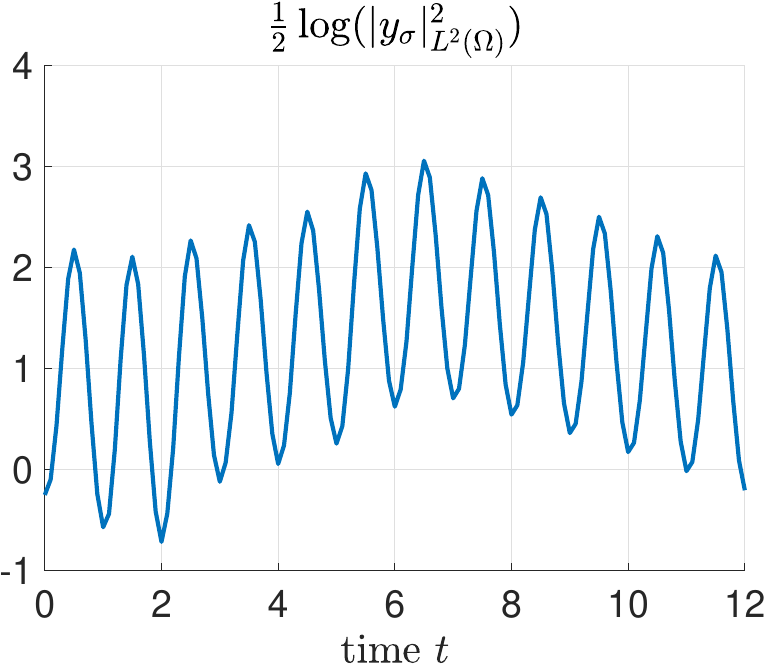}}
\qquad
\subfigure
{\includegraphics[width=0.45\textwidth,height=0.375\textwidth]{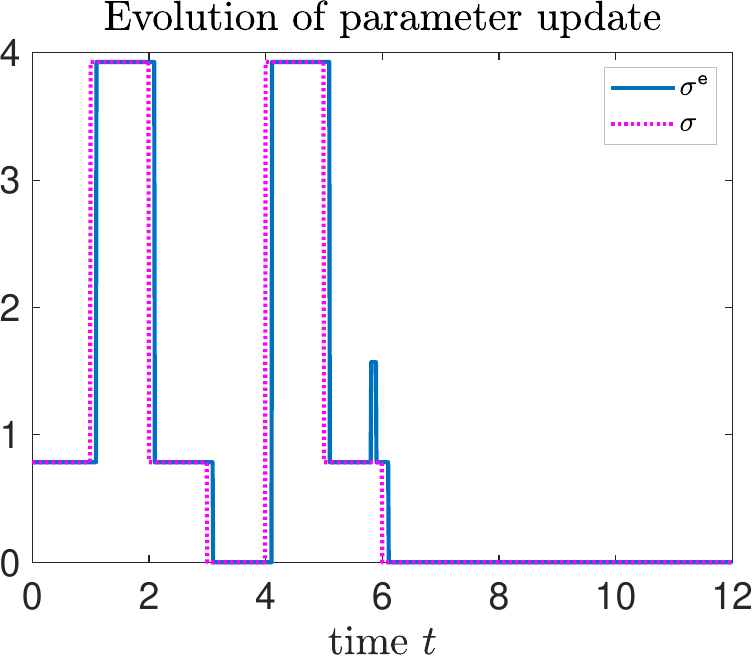}}
\caption{Piecewise constant, quickly switching, parameters $\sigma(t)\in\varSigma$.}
\label{fig.PDE_sigma_t_real}
\end{figure}
 Furthermore, in Fig.~\ref{fig.PDE_sigma_t_real} we have taken the true (piecewise constant) parameter~$\sigma(t)$ with values in the training set, and with the exact/correct initial guess. We see that our strategy provides us with the correct values for~$\sigma(t)$ in every but one updating time interval~$I_{j_0}^\tau\subset(5,6)$. In spite of the correct identification, we can  see that the controlled dynamics shows instability up to time~$t=6$, when the  parameter is switching fast. Thus, in this example, even with exact identification of~$\sigma(t)$ we may obtain an unstable system (cf.~\cite[Prop.~1]{AkarPaulSafoMitra06} and~\cite[Prob.~A]{LiberzonMorse99}). Therefore, in the case of fast switching parameters our strategy will fail. To avoid this we will likely need to include the dwell time(s) as part of the uncertainty. Finally, note that any switching system between the dynamics given by the elements in~$\{\clA_\varsigma^K\mid \varsigma\in\varSigma\}$ will be stable for large enough dwell time~$T$, indeed, if every~$\clA_\varsigma^K$ is $(C,\mu)$-stable with~$C\ge1$ and~$\mu>0$ (cf.~Def.~\ref{D:stable}), then we can see that any switching system will be stable for~$T>\mu^{-1}\log(C)$, because at time~$t=jT$ we will have~$\norm{y(jT)}{H}\le C\rme^{-\mu T }\norm{y((j-1)T)}{H}\le C^j\rme^{-\mu jT }\norm{y(0)}{H}=(C\rme^{-\mu T })^j\norm{y(0)}{H}$.

\section{Comparison to robust a-priori feedback}
Let us revisit the finite-dimensional time-periodic dynamics from Section~\ref{sS:num-per} in the special case~$\sigma = (1,\phi)$, i.e., with~$\rho =1$, and
\begin{equation}\label{eq:robA}
\clA_{\sigma}(t) = \clA_{(1,\phi)}(t) = \Psi(t+\phi) \begin{bmatrix} 0 &1 \\ 1& 0\end{bmatrix},\quad\mbox{with}\quad \Psi(s) = 1+6\sin(2\pi s),
\end{equation}
and with uncertain time-phase~$\phi \in [0,1)$. Following the approach in \cite{GuthKunRod23-arx}, we construct a robust a priori feedback~$K_{\varSigma}=-B^{\top} \clE^{\top} \bfPi_{\varSigma} \clE$, based on a finite ensemble of candidate parameter values~$\varSigma = (\sigma_i)_{i=1}^N$. Here~$\clE:\bbR^n \to \bbR^{nN}$,~$x \mapsto \begin{bmatrix} x^\top, x^\top, \ldots, x^\top \end{bmatrix}^\top$ and~$\bfPi_\varSigma \in \bbR^{nN\times nN}$ solves the~$1$-periodic Riccati equation
\begin{equation}\label{eq:extRiccati}
\dot\bfPi_{\varSigma}+\clA_{\varSigma}^*\bfPi_{\varSigma}+\bfPi_{\varSigma} \clA_{\varSigma}-\bfPi_{\varSigma}\clE B B^*\clE^\top \bfPi_{\varSigma} +\frac{1}{N}\clE C^*C \clE^\top=0,\quad\bfPi_{\varSigma}(t+1)=\bfPi_{\varSigma}(t),
\end{equation}
where~$B$ and~$C$ are chosen as in~\eqref{BCQ-ode-per}, and~$\clA_\varSigma$ is a block diagonal matrix containing the matrices~$(\clA_\sigma)_{i=1}^N$.

We compare the robust feedback~$K_{\varSigma}$ to the feedback~$K_\sigma$ corresponding to the true parameter as well as to the feedback~$K_{\sigma^\tte}$ provided by the adaptive approach (cf.~Sect.~\ref{sS:num-per}) and present results using $N=8$ candidate parameters collected in the ensemble~$\varSigma = \left\{\left(1,~0.05 + \frac{i-1}{N} \right)\mid 1\le i \le N\right\}$ and true time-phase~$\phi= 0.51$. For the adaptive approach we choose the update time as $\tau = 0.2$ and initial guess~$\sigma^\tte(0)=(1,0.05)$. In each time interval $\clI_j^\tau = ((j-1)\tau,j\tau)$ we take $\varSigma^* = \varSigma_l \cup \varSigma_g$, with local parameters $\varSigma_l \subseteq \varSigma$ is a ball with radius $\gamma = 0.1$ centered around the latest update and with no random global parameters, $N_g = 0$; see Algorithm~\ref{Alg:Onl-concat}.

In this example the stabilization rate of the robust feedback~$K_\varSigma$ is faster than the stabilization rate of the optimal feedback~$K_\sigma$ corresponding to the true parameter~$\sigma=(1,0.51)$, which in turn is faster than the stabilization rate of the adaptive feedback~$K_{\sigma^{\tt e}}$ (see, Fig.~\ref{fig:robust}). As expected, the feedback~$K_\sigma$ leads to the smallest cost: let~$y_\sigma$ denote the states controlled using the feedback~$K_\sigma$, let~$y_{\sigma^{\tt e}(t)}$ be the states controlled using the adaptive feedback~$K_{\sigma^{\tt e}(t)}$, and let~$y_\varSigma$ be the states controlled using the robust feedback~$K_\varSigma$, then we obtain 
\begin{align*}
\int_0^{20} \left(\|Cy_{\sigma}(t)\|_{\bbR}^2 +  \|K_\sigma y_{\sigma}(t)\|_{\bbR}^2\right)\rm dt \approx 0.0463,\\
\int_0^{20} \left(\|Cy_{\sigma^{\tt e}(t)}(t)\|_{\bbR}^2 + \|K_{\sigma^{\tt e}(t)}\, y_{\sigma^{\tt e}(t)}(t)\|_{\bbR}^2\right) \rm dt \approx 1.1766,\\
\int_0^{20} \left(\|Cy_{\varSigma}(t)\|_{\bbR}^2 + \|K_\varSigma\, y_{\varSigma}(t)\|_{\bbR}^2\right) \rm dt \approx 2.5197. 
\end{align*}
\begin{figure}[ht]
\centering
\subfigure
{\includegraphics[width=0.45\textwidth,height=0.375\textwidth]{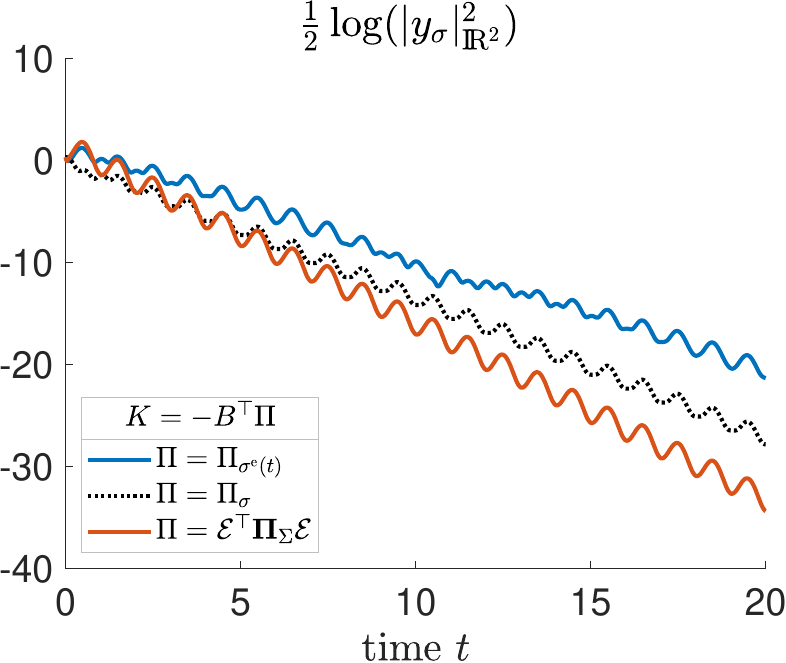}}
\qquad
\subfigure
{\includegraphics[width=0.45\textwidth,height=0.375\textwidth]{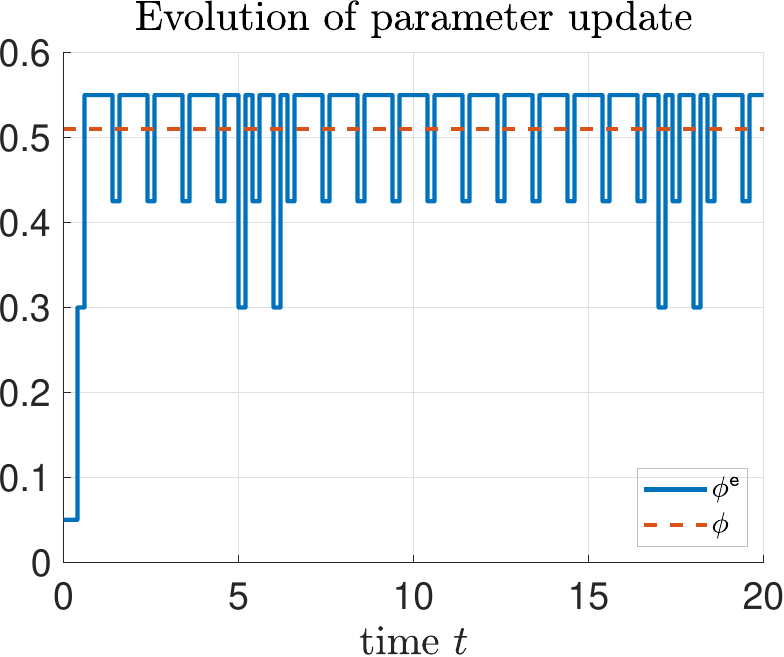}}
\caption{Comparison of adaptive, optimal, and robust feedback with $\clA_\sigma$ as in~\eqref{eq:robA}. }
\label{fig:robust}
\end{figure}
The parameter update is shown in Fig.~\ref{fig:robust}.  Comparing to section~\ref{sS:num-per}, we see that, in the present situation where we assume that the time-period is known, the algorithm performs better concerning the  identification of the true parameter.
\subsection{Stabilizability and detectability of ensembles}
In order to ensure the existence of a unique solution~$\bfPi_\varSigma$ of \eqref{eq:extRiccati} we recall from \cite[Thm.~3]{SilvermanMeadows67} that~$(\clA_\varSigma,\clE B)$ is controllable (and thus stabilizable) on a nonempty open time interval~$(t_0,t_1)$ if~$Q_B(t)$ has full rank for some time~$ t\in(t_0,t_1)$, where
\begin{equation*}
Q_B(t) \coloneqq \begin{bmatrix} P_0(t) & P_1(t)& \ldots P_{nN-1}(t) \end{bmatrix},
\end{equation*}
with 
\begin{equation*}
P_{k+1}(t) \coloneqq -\clA_{\varSigma}(t) P_k(t) + \dot{P}_k(t),\quad P_0(t) = \clE B.
\end{equation*}
Similarly, from \cite[Thm.~5]{SilvermanMeadows67} we know that~$(\clA_\varSigma,C \clE^\top)$ is observable (and thus detectable) on~$(t_0,t_1)$ if $Q_C(t)$ has full rank for some time~$ t\in(t_0,t_1)$, where
\begin{equation*}
Q_C(t) \coloneqq \begin{bmatrix} S_0(t) & S_1(t)& \ldots S_{nN-1}(t) \end{bmatrix},
\end{equation*}
with 
\begin{equation*}
S_{k+1}(t) \coloneqq \clA^\ast_{\varSigma}(t) S_k(t) + \dot{S}_k(t),\quad S_0(t) = \clE C^\ast.
\end{equation*}
Let us consider, for instance, the case $N=2$. Denoting $\varphi_i(t) \coloneqq \Psi(t+\phi_i)$, $i \in \{1,2\}$ we have
\begin{equation*}
Q_B(t) = \begin{bmatrix} 
0 & -\varphi_1(t) & -\frac{\rm d}{\rm dt} \varphi_1(t) & -\varphi_1^3(t) - \frac{\rm d^2}{\rm dt^2} \varphi_1(t)\\
1 & 0 & \varphi_1^2(t) & \varphi_1(t) \frac{\rm d}{\rm dt} \varphi_1(t) + \frac{\rm d}{\rm dt} \varphi_1^2(t) \\ 
0 & -\varphi_2(t) & -\frac{\rm d}{\rm dt} \varphi_2(t) & -\varphi_2^3(t) - \frac{\rm d^2}{\rm dt^2} \varphi_2(t)\\
1 & 0 & \varphi_2^2(t) & \varphi_2(t) \frac{\rm d}{\rm dt} \varphi_2(t) + \frac{\rm d}{\rm dt} \varphi_2^2(t) 
\end{bmatrix},
\end{equation*}
and
\begin{equation*}
Q_C(t) = \begin{bmatrix}
1 & 0 & \varphi_1^2(t) & \varphi_1(t) \frac{\rm d}{\rm dt} \varphi_1(t) + \frac{\rm d}{\rm dt} \varphi_1^2(t) \\
0 & \varphi_1(t) &   \frac{\rm d}{\rm dt} \varphi_1(t)  & \varphi_1^3(t) + \frac{\rm d^2}{\rm dt^2} \varphi_1(t) \\
1 & 0 & \varphi_2^2(t) & \varphi_2(t) \frac{\rm d}{\rm dt} \varphi_2(t) + \frac{\rm d}{\rm dt} \varphi_2^2(t)\\
0 & \varphi_2(t) & \frac{\rm d}{\rm dt} \varphi_2(t) & \varphi_2^3(t) + \frac{\rm d^2}{\rm dt^2} \varphi_2(t) 
\end{bmatrix}.
\end{equation*}
For the particular choice $\phi_1 = 0$ and $\phi_2 = 0.5$ we obtain
\begin{equation*}
Q_B(0.5) =  \begin{bmatrix}
0 & -1 & 12\pi &  -1\\
1 & 0 & 1 & -36 \pi\\
0 & -1 & -12\pi &  -1\\
1 & 0 & 1 & 36\pi
\end{bmatrix},\qquad Q_C(0.5) =  \begin{bmatrix}
1 & 0 & 1  & -36\pi \\
0 & 1 & -12\pi & 1\\
1 & 0 & 1 & 36\pi \\
0 & 1 & 12\pi & 1
\end{bmatrix},
\end{equation*} 
which both have full rank. Thus, for~$\varSigma = \{(1,0),(1,0.5)\}$, the system~$(\clA_{\varSigma},\clE B)$ is controllable and~$(\clA_\varSigma,C \clE^\top)$ is observable on nonempty open intervals containing~$t=0.5$, and thus \eqref{eq:extRiccati} admits a unique solution. Verifying that $Q_B$ and $Q_C$ have full rank for some~$t$ quickly becomes more complex as the number of candidate parameters~$N$ increases. The investigation of the general case with~$N>2$ candidate parameters is beyond the scope of this manuscript, and thus left for future research.

\section{Final remarks}\label{S:finRemks}
We have shown that if we select a finite subset of training parameters~$\varSigma\subseteq\fkS$, which is ``dense enough'' in~$\fkS$, and store, offline, the Riccati feedback inputs~$K_\varsigma$ corresponding to each parameter in~$\varsigma\in\varSigma$, then we can construct an input control, stabilizing the system for every parameter~$\sigma\in\fkS$. For this purpose we take one of the stored feedbacks~$K_\varsigma$ in time intervals~$I_j^\tau=((j-1)\tau,j\tau)$, $j\in\bbN_+$, and update~$\varsigma$, online, based on input and output data comparison in the same time intervals. This requires solving auxiliary fictitious systems online corresponding to some values in~$\varSigma$ (candidates for the true unknown~$\sigma$), which can be done in parallel (while the true plant/system is running).

\subsection{Piecewise constant~$\sigma$} We focused  in the case of constant parameters. Besides, we have seen that the strategy is able to respond to changes in the true parameter~$\sigma$ and can be applied to piecewise constant parameters~$\sigma$ as well. The obtained feedback input will be stabilizing if the intervals of constancy are long enough.

\subsection{Other types of uncertainties.}
We have focused on uncertainties in the parameter~$\sigma$. In future works, it would be important to address the robustness against other types of uncertainties, disturbances, and noisy measurements.  It would be of interest, to investigate the problem of designing an observer to estimate the state of the system and to check the robustness of the approach when coupled with such an observer.

\subsection{Worst case scenarios} Can we find~$\varsigma\in\fkS$ such that the Riccati feedback input~$K_\varsigma$ will make~$\clA_\sigma+BK_\varsigma$ exponentially stable, for all~$\sigma\in\fkS$?
In Example~\ref{sS:num-osc}, we can indeed find such a~$\varsigma$, namely, $\varsigma=1\in\fkS=[-1,1]$. In fact,~$\clA_\varsigma+BK_\varsigma$ is stable only if $K_\varsigma=\begin{bmatrix}\kappa_1&\kappa_2\end{bmatrix}$ satisfies~$\kappa_1-1<0$ and~$\kappa_2+\varsigma<0$. Thus, $\kappa_2+\sigma<0$, for all $\sigma\in\fkS$, and~$\clA_\sigma+BK_\varsigma$ is stable.
On the other hand, in Example~\ref{sS:num-PDE}, for the uncertainty in the convection direction~$b_\sigma$, $\sigma\in[0,2\pi)$, we do not expect that such a direction~$\varsigma$ exists  for the considered optimal Riccati based feedback inputs~$K_\varsigma$.

If we are given a sufficiently large  number of actuators and if we are not concerned with minimizing the spent energy, then we could consider feedback inputs other than Riccati based ones, for example, as in~\cite{Rod21-aut}. To guarantee that the feedbacks in~\cite{Rod21-aut} are stabilizing, we know that, with respect to the convection term, the number of needed actuators depends on an upper bound for a suitable norm of the convection operator, which translates, in our case, to an upper bound for the norm of the vector~$b_\sigma\in\bbR^2$. This latter norm is bounded and we conclude that an explicit feedback as in~\cite{Rod21-aut} will be able to stabilize the system for all~$\sigma\in\fkS=[0,2\pi)$.  In this manuscript we are interested in Riccati based feedbacks, which succeed to stabilize the system, while explicit ones fail if the number of actuators is too small.

\bigskip\noindent
{\bf Aknowlegments.}
S. Rodrigues gratefully acknowledges partial support from
the State of Upper Austria and Austrian Science
Fund (FWF): P 33432-NBL.

{
\bibliographystyle{plainurl}
\bibliography{Uncertain_OnOff}
}

\end{document}